\documentclass[11pt]{article} 

\usepackage{geometry} 
\usepackage{graphicx} 

\usepackage{booktabs} 
\usepackage{array} 
\usepackage{paralist} 
\usepackage{verbatim} 
\usepackage{subfig} 
\usepackage{amsmath}
\usepackage{amssymb}
\usepackage{amsthm}
\usepackage{mathtools}
\usepackage{bbm}

\usepackage{fancyhdr} 
\numberwithin{equation}{section}

\theoremstyle{plain}
\newtheorem{thm}{Theorem}[section]
\newtheorem{lem}[thm]{Lemma}
\newtheorem{prop}[thm]{Proposition}
\newtheorem{cor}[thm]{Corollary}
\theoremstyle{definition}
\newtheorem{defn}[thm]{Definition}
\newtheorem{exmp}[thm]{Example}
\newtheorem{ass}[thm]{Assumption}
\newtheorem{rem}[thm]{Remark}

\def\Ccal{\mathcal{C}}
\def\Dcal{\mathcal{D}}
\def\Ecal{\mathcal{E}}

\def\Gcal{\mathcal{G}}
\def\Hcal{\mathcal{H}}
\def\Lcal{\mathcal{L}}
\def\Mcal{\mathcal{M}}
\def\Scal{\mathcal{S}}

\def\Dbb{\mathbb{D}}

\def\Nbb{\mathbb{N}}
\def\Pbb{\mathbb{P}}
\def\Rbb{\mathbb{R}}
\def\Wbb{\mathbb{W}}

\def\1bb{\mathbbm{1}}

\DeclareMathOperator{\dis}{dis}
\DeclareMathOperator{\diam}{diam}

\DeclareMathOperator{\Tr}{Tr}

\let\epsilon\varepsilon

\title{Degenerate limits for one-parameter families \\ of non-fixed-point diffusions on fractals}
\author{Ben Hambly\footnote{Mathematical Institute, University of Oxford, Woodstock Road, Oxford, OX2 6GG, UK. Email: hambly@maths.ox.ac.uk.}\ \ and Weiye Yang\footnote{Mathematical Institute, University of Oxford, Woodstock Road, Oxford, OX2 6GG, UK. Email: weiye.yang@maths.ox.ac.uk. ORCiD: 0000-0003-2104-1218.}}
\date{} 

\begin{document}
\maketitle

\begin{abstract}
The Sierpinski gasket is known to support an exotic stochastic process called the asymptotically one-dimensional diffusion. 
This process displays local anisotropy, as there is a 
preferred direction of motion which dominates at the microscale, but on the macroscale we see global isotropy in that the process will 
behave like the canonical Brownian motion on the fractal. In this paper we analyse the microscale behaviour of such processes, 
which we call non-fixed point diffusions, for a class of fractals and 
show that there is a natural limit diffusion associated with the small scale asymptotics. This limit process no longer lives on the original 
fractal but is supported by another fractal, which is the Gromov-Hausdorff limit of the original set after a shorting operation is 
performed on the dominant microscale direction of motion. We establish the weak convergence of the rescaled diffusions in a general 
setting and use this to answer a question raised in \cite{hattori1994} about the ultraviolet limit of the asymptotically one-dimensional 
diffusion process on the Sierpinski gasket.
\end{abstract}

\section{Introduction}

The study of diffusion on fractals has largely focused on constructing and analysing the `Brownian motion', that is the stochastic
process generated by the `natural' Laplace operator, on a given fractal. For the Sierpinski gasket, 
the simple symmetric random walk on the natural graph approximation has the property of being decimation invariant, in that it has 
same law when stopped at its visits to coarser approximations and this enabled the initial detailed analysis of the diffusion
and its properties \cite{barlow1988}. However the simple symmetric random walk is not the only possible discrete Markov chain on 
graph approximations to the Sierpinski gasket that can be used to construct a scaling limit. By considering processes invariant under 
reduced symmetry groups it is possible to construct processes such as the rotationally and 
scale invariant but non-reversible $p$-stream diffusions of \cite{kumagai1995}, the not-necessarily scale invariant homogeneous diffusions 
of \cite{heck1998} and the one that will provide a fundamental example for our work, the asymptotically one-dimensional diffusion 
of \cite{hattori1994}. This process is invariant under reflection in the vertical axis but is not scale invariant and displays local anisotropy but 
global isotropy. 

The construction of a canonical Brownian motion on the Sierpinski gasket
was generalized to nested fractals, \cite{lindstrom1990}, through to the large class of finitely ramified fractals, the p.c.f.~self-similar 
sets of \cite{Kigami2001}. In these extensions it became clear that a straightforward approach to the construction of a Brownian motion 
was available through the theory of Dirichlet forms, \cite{fukushima92}, \cite{kusuoka89}, \cite{kigami93}, and the questions of the existence and 
uniqueness of the Brownian motion could 
be reduced to the existence and uniqueness of a fixed point for a finite dimensional renormalization map on the cone of Dirichlet 
forms over the basic cell structure in the fractal, \cite{metz95}, \cite{sabot1997}. When seeking to generalize some of the exotic diffusions
on the Sierpinski gasket it is easiest to work in the reversible case (which excludes the $p$-stream and homogeneous difusions) 
and use the theory of Dirichlet forms. Using this approach the asymptotically one-dimensional diffusion processes were extended to 
some nested fractals in \cite{hambly1998}. Our first aim here will be to generalize this class further to a sub-class of p.c.f.~self-similar
fractals. 

The construction of a Laplace operator on a finitely ramified fractal through an associated sequence of Dirichlet forms just requires that 
a compatible sequence of resistance neworks can be constructed on the graph approximations to the fractal \cite{Kigami2001}. 
The self-similarity allows this to be reduced to the study of a finite dimensional renormalization map on a cone of discrete irreducible 
Dirichlet forms.  The fixed point problem for this renormalization map was the subject of the work of Sabot, Metz and Pierone, see (among others) 
\cite{sabot1997}, \cite{metz95}, \cite{metz1996}, \cite{peirone2014}, \cite{peirone2013}. In solving the uniqueness problem for nested fractals 
\cite{sabot1997} and \cite{metz1996} showed that when considering the renormalization map, although
it is not in gerenal a strict contraction in Hilbert's projective metric on the cone, iterates of the map converge to a non-degenerate 
fixed point under some conditions. Essentially one has to examine the map in the neighbourhood of possible degenerate fixed points on the boundary of the cone of irreducible Dirichlet forms and find conditions that ensure 
there is a move `away' from degeneracy and hence the map can be iterated toward the non-degenerate fixed point. The construction of what we will call here
a one-parameter family of non-fixed point diffusions is then about associating a sequence of resistor networks with the inverse iterates of 
this renormalization map, within a one-parameter family, so that the corresponding networks resemble the degenerate fixed point on the 
small scale but resemble the non-degenerate fixed point on the large scale. Using this approach, \cite{hattori1994}, \cite{hambly1998} 
showed that it was possible to construct a one parameter family of diffusions on the Sierpinski gasket by inverting a one-parameter 
version of the renormalization map and placing suitable conductances on the graphs $G_n$ which approximate the gasket so that $G_0$, the graph given by the triangle, had the effective resistance 
given by $(1,w,w)$. Figure~\ref{fig:oldgas} shows the configuration of resistors and the renormalization of total resistance $R(w) = (12w^2+26w+12)/(w^3+8w^2+15w+6)$ and for the diagonal edges $\gamma(w)=(6w^2+4w)/(w^2+6w+3)$ obtained from the renormalization map.

\begin{figure}[ht]
\centerline{\includegraphics[height=1in]{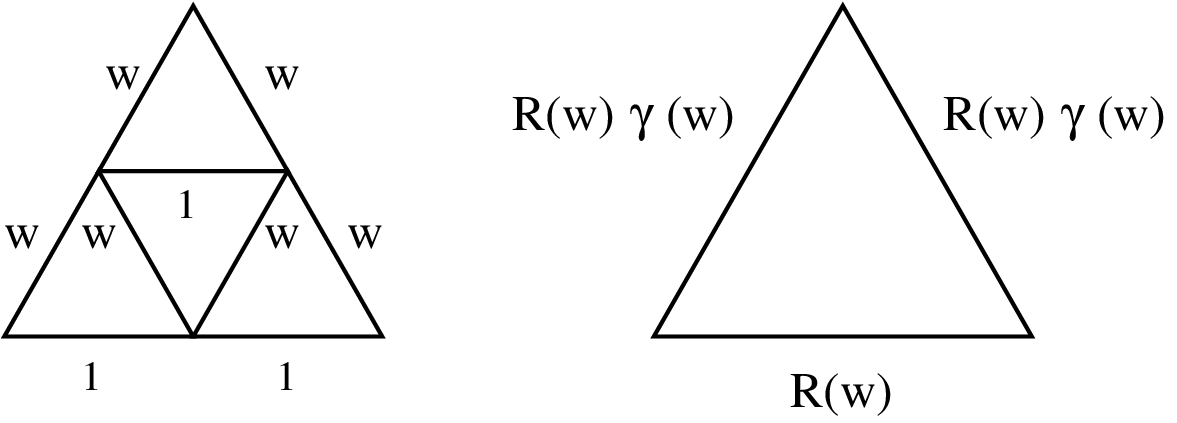}}
\caption{The original weighting of the edges of the Sierpinski gasket for the asymptotically one-dimensional process and the renormalization.}
\label{fig:oldgas}
\end{figure}

Let $\{X^{n,w}_t;t\geq 0\}$ denote the continuous time random walk corresponding to the sequence of resistance networks $G_n$
in which the effective conductance over $G_0$ is given by the triple of conductors $(1,w,w)$. Then, it was shown in \cite{hattori1994}, that 
for any starting weight $0<w< 1$, as $n\to\infty$, 
\[ \{ X^{n,w}_{6^nt};t\geq 0\} \to \{X^{a,w}_t; t\geq 0\}, \]
weakly in $\Dcal_G[0,\infty)$ (the space of c\`adl\`ag functions on the limiting fractal $G$) where $X^{a,w}$ is the asymptotically one-dimensional 
diffusion, a continuous, strong Markov process on $G$. This is called an asymptotically one-dimensional diffusion as the local behaviour of the 
conductances gives the horizontal line segment an increasing weight relative to the two diagonal segments. 

To see the large scale behaviour of this diffusion we can extend the graph $G_0$ to an infinite Sierpinski gasket graph $\tilde{G}_0$, 
in which each copy of the basic triangle has conductance $(1,w,w)$, and proceed to construct
the limiting asymptotically one-dimensional diffusion $\{\tilde{X}^{a,w}_t;t\geq 0\}$ on the infinite fractal $\tilde{G}$.  Then, in \cite{hambly1998}, it is shown that under
the classical scaling for the Brownian motion on $\tilde{G}$ we have, as $n\to\infty$,  
\[ \{2^{-n}\tilde{X}^{a,w}_{5^nt}; t\geq 0\} \to \{B_t;t\geq 0\}, \]
weakly in $C_{\tilde{G}}[0,\infty)$ (the space of continuous functions on $\tilde{G}$), where $B$ is the Brownian motion on $\tilde{G}$.
This is a homogenization result not seen in Euclidean space as the geometry of the fractal causes the homogenization. 
If we think of a diffusing particle, even though it moves much more frequently in a horizontal direction, in order to cross large 
regions it must make vertical moves and thus on the very large scale it behaves like the Brownian motion.

In \cite{hattori1994} it is observed that, even in the situation where there is no fixed point for the renormalization map, 
this procedure could lead to a diffusion on the $O(1)$ scale which was non-degenerate even though at both small and large scale the diffusion 
is degenerate. This was illustrated in the case of $abc$-gaskets where the ratios between $a,b$ and $c$ are such that there is no fixed point 
\cite{HW97}. 

It is clear that for the asymptotically one-dimensional diffusion on the Sierpinski gasket on the small scale there is a separation of time 
scales in that there will be many more horizontal steps than vertical ones. The question of the so-called 
ultraviolet limit of these processes was raised in \cite{hattori1994}. Is there a rescaling of the diffusion process over short time 
and space scales which will produce a non-trivial object in the short scale limit? 
In other words does there exist a scaling $\lambda \in (0,1)$ and a non-trivial process $X^b$ such that
\[ \{ 2^n X^{a,w}_{\lambda^n t}; t\geq 0\} \to \{X^b_t; t\geq 0\} \]
weakly as $n\to\infty$? Our aim in this paper is to consider a class of 
fractals for which there are asymptotically lower-dimensional processes and discuss their short time
scaling limits. The degenerate diffusions which arise could be called the ultraviolet limits for the asymptotically lower-dimensional diffusions.

In order to examine the short time scaling limit we show that, by thinking of the fractal as a resistance form, a metric space equipped 
with a resistance metric, and shorting the high conductance edges, there is Gromov-Hausdorff convergence of the fractals 
to a limit space. We can then exploit the recent work of \cite{croydon2016a} (extending that in \cite{croydon2016}) to establish that there 
is weak convergence of the rescaled diffusions to a diffusion process on this limit fractal. This limit fractal is not a p.c.f.~self-similar 
set, but is typically a simple fractal space \cite{nyberg95}, and the theory for p.c.f.~fractals is easily extended to include such limit objects. 
We note that this limit construction provides a case where we fuse each element of a countable collection of subsets of a space, going beyond 
the fusing of a single subset or finite pairs of points as discussed in \cite{kigami2012}, \cite{croydon2012}, \cite{croydon2016a}. 

In our approach to these non-fixed point diffusions we write the conductance matrix in a different form to those
used previously in that we fix (in Figure~\ref{fig:oldgas}) what was $w$ to be 1 and let the edge which was 1 have conductance 
$v$ for $v>1$. Examples are shown in Figures~\ref{fig:gasblob} and~\ref{fig:vicsek}.
This shows that, for the Sierpinski gasket, as we look over smaller scales the horizontal edge has conductance
which goes to infinity and thus in the limit we expect that the horizontal edges will be shorted, leading to an object which is shown 
on the right side of Figure~\ref{fig:gasblob}. The analysis of the fixed point problem in \cite{sabot1997} required the analysis of such
a shorted graph. Here we establish a general result about the weak convergence of diffusions on a class of 
fractals as such a parameter $v$ tends  to $\infty$. We work in the framework of (generalized) p.c.f.~self-similar sets but restrict the class 
to those which support a resistance form, and whose shorted versions also support a resistance form. We will view our sequence of 
fractals with the resistance metric as metric spaces and embed them along with the limit into a metric space $\mathcal M$. We 
prove the weak convergence in $\Dcal_\Mcal[0,\infty)$, the space of c\`{a}dl\`{a}g processes on $\Mcal$.

%

We can use our main result, Theorem~\ref{thm:main}, in the case of the Sierpinski gasket to analyse the short time behaviour of 
the asymptotically one-dimensional diffusion. We will consider the diffusion over small scales and times and prove the weak convergence 
to a diffusion on the shorted gasket, the limit of the graphs on the right hand side of Figure~\ref{fig:gasblob}. We will think of the 
gasket $G$ now as a self-sufficient metric space with a resistance metric $R^v$ determined by the conductance parameter $v$. 
Let $\psi_1,\psi_2,\psi_3$ denote the three similitudes used in defining the Sierpinski gasket and we assume that they have fixed 
points $p_1$, $p_2$ and $p_3$, respectively. We assume that for $G_0$, the triangle graph on the points $p_i$, the edge from $p_1$ 
to $p_3$ is of resistance $v$ and the other two edges are of resistance 1. We write $T_A(X)=\inf\{t> 0: X_t \in A\}$ for the hitting 
time of a set $A$ by the process $X$. We write $L_n$ for the image of the line joining $p_1, p_3$ in the triangle $\psi_2^{n}(G_0)$. That is the 
bottom edge of the triangle with Euclidean size $2^{-n}$ with top vertex at $p_2$.

\begin{thm}\label{thm:sg}
Let $X^{a,v}$ denote the asymptotically one-dimensional diffusion on the metric space $(G,R^v)$ with conductances $(1,1,v)$ on $G_0$ 
and $X^{a,v}_0=p_2$ and let  $X^s$ be the diffusion on the `shorted gasket', that is the limit of the graphs shown in 
Figure~\ref{fig:gasblob}, with $X^s_0=0$, where $0$ denotes the top vertex with only one edge into it and the 
base vertex is denoted by $1$. Then, there exists a constant $\sigma>0$ such that 
\[  \{ \psi_2^{-n}(X^{a,v}_{(9/2)^{-n}t});0\leq (9/2)^{-n}t\leq T_{L_n}(X^{a,v}) \} \to \{X^s_{\sigma t};
0\leq t\leq T_1(X^s)\}, \]
weakly in $\Dcal_\Mcal[0,\infty)$.
\end{thm}

\begin{figure}[ht]
\centerline{\includegraphics[height = 2in]{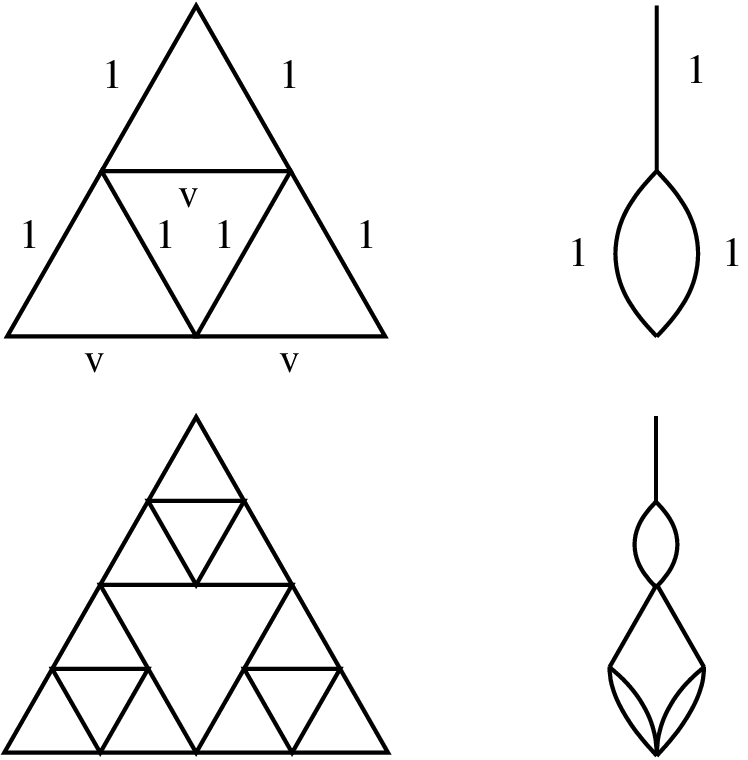}}
\caption{The first two levels of the Sierpinski gasket and the corresponding levels of the limit fractal obtained as $v\to\infty$.}
\label{fig:gasblob}
\end{figure}

In the setting in which we work the limit spaces will have an associated resistance form and it is straightforward to determine some of
the properties of the limit diffusion. In particular it is not difficult to see that the spectral dimension for the shorted gasket 
is $2\log{3}/\log{9/2}$. In the paper \cite{hambly2002} this was shown to be the local spectral dimension (the exponent for the 
short time asymptotic decay of the on-diagonal heat kernel) for the asymptotically one-dimensional diffusion on the Sierpinski gasket.

The outline of the paper is as follows. We begin by giving the framework in which we will work in Section~2. As the shorted gasket is
not a p.c.f.~self-similar set we will need a small extension of the class of p.c.f.~self-similar sets to discuss our limit processes. 
In Section~3 we introduce the class of locally degenerate diffusions on fractals. We then construct the limit spaces that will support 
our limiting diffusions in Section~4. 
In order to prove our result we use Gromov-Hausdorff-vague convergence and hence we construct a suitable space in which to embed the 
sequence of fractals and the limiting space. We do this in Section~5 and establish the weak convergence by showing how to employ the 
result of \cite{croydon2016a} (simplifying substantially our original 30 page direct proof of the weak convergence).

\section{Preliminaries}

We slightly generalize the idea of a post critically finite self-similar set \cite{Kigami2001}. A generalization which goes beyond what we 
introduce here can be found in \cite{tep2008}. 

\begin{defn}
Fix some $N \in \Nbb$. Let $F = (F,d)$ be a metric space and for $i = 1,\ldots,N$ let $\psi_i: F \to F$ be a function. 
Then $\Scal = (F,(\psi_i)_{1 \leq i \leq N})$ is a \textit{self-similar structure} if
\begin{enumerate}
\item $(F,d)$ is compact,
\item The $\psi_i$ are continuous injections from $F$ to itself with
\begin{equation*}
F = \bigcup_{i=1}^N \psi_i (F),
\end{equation*}
\item There exists $\delta > 0$ such that
\begin{equation*}
d(\psi_i(x),\psi_i(y)) \leq (1-\delta) d(x,y)
\end{equation*}
for all $x,y \in F$ and $1 \leq i \leq N$.
\end{enumerate}
\end{defn}

Let $I := \{ 1,\ldots, N \}$, and define the word spaces $\Wbb_n = I^n$, $\Wbb = I^\Nbb$. Endow $\Wbb$ with the product topology, and 
let $\sigma$ be the left shift operator, which maps $\Wbb$ to $\Wbb$ or $\Wbb_n$ to $\Wbb_{n-1}$. That is, if $w = w_1w_2w_3\ldots$ then
$\sigma(w) = w_2w_3\ldots$.
For $w = w_1 \ldots w_n \in \Wbb_n$ we now define
\begin{equation*}
\psi_w := \psi_{w_1} \circ \cdots \circ \psi_{w_n}.
\end{equation*}
For $A \subseteq F$ let $A_w := \psi_w(A)$. If $n \geq 1$ and $w \in \Wbb$ (or $w \in \Wbb_m$ with $m \geq n$) then define
\begin{equation*}
w|n = w_1 \ldots w_n \in \Wbb_n.
\end{equation*}
For each $w \in \Wbb$, there exists $x_w \in F$ such that
\begin{equation*}
\bigcap_{n=1}^\infty \psi_{w|n}(F) = \{ x_w \}.
\end{equation*}
and we write $\pi: \Wbb \to F$ for the mapping $\pi(w) := x_w$. As in \cite[Lemma 5.10]{barlow1998}, we have the following result:

\begin{lem}\label{pimap}
The function $\pi$ is the unique mapping $\pi: \Wbb \to F$ such that for all $w \in \Wbb$ and $i \in I$,
\begin{equation*}
\pi(i \cdot w) = \psi_i(\pi(w))
\end{equation*}
where $i \cdot w \in \Wbb$ denotes the concatenation of $i$ with $w$. Moreover, $\pi$ is continuous and surjective.
\end{lem}


\begin{defn}
For a self-similar structure $\Scal = (F,(\psi_i)_{1 \leq i \leq N})$, let
\begin{equation*}
B(\Scal) = \bigcup_{i,j,i\neq j} F_i \cap F_j.
\end{equation*}
This is the set of points that exist in the ``overlap'' of the images of two distinct $\psi_i$. Let
\begin{equation*}
\Gamma(\Scal) = \pi^{-1}(B(\Scal))
\end{equation*}
be the set of words corresponding to $B(\Scal)$. This is called the \textit{critical set} of $\Scal$. Let
\begin{equation*}
P(\Scal) = \bigcup_{n=1}^\infty \sigma^n(\Gamma(\Scal))
\end{equation*}
be the \textit{post-critical set} of $\Scal$.
\end{defn}
\begin{defn}[Generalized p.c.f.s.s.~sets]\label{genPCF}
A self-similar structure $(F,(\psi_i)_{1 \leq i \leq N})$ is called \textit{generalized post-critically finite} if $\pi(P(\Scal))$ is finite.
A metric space $(F,d)$ is a \textit{generalized post-critically finite self-similar} set, or \textit{generalized p.c.f.s.s.}~set, if there exists 
a generalized post-critically finite self-similar structure $(\psi_i)_{1 \leq i \leq N}$ on $F$.
\end{defn}
Now $\pi(P(\Scal))$ has two equivalent reformulations given below:
\begin{equation*}
\begin{split}
\pi(P(\Scal)) &= \left\{ x \in F: \exists w,v \in \bigcup_{n \geq 1} \Wbb_n,\ w \neq v,\ \psi_w(x) \in F_v \right\}\\
&= \left\{ x \in F: \exists w \in \bigcup_{n \geq 1} \Wbb_n,\ \psi_w(x) \in B(\Scal) \right\}.
\end{split}
\end{equation*}

\begin{rem}
Notice that Definition \ref{genPCF} differs from the definition of a p.c.f.s.s.~set (see \cite{Kigami2001}), which has the stronger condition of $P(\Scal)$ itself 
being finite.
\end{rem}

\begin{exmp}
Any p.c.f.~self-similar set is clearly a generalized p.c.f.s.s.~set. Also the `shorted gasket' of Figure~\ref{fig:gasblob} and the Diamond 
Hierarchical Lattice studied in \cite{hambly2010} are examples of generalized, but not strictly, p.c.f.s.s.~sets.
\end{exmp}

\begin{defn}
Let $(F,(\psi_i)_{1 \leq i \leq N})$ be a generalized p.c.f.s.s.~set. For $n \geq 0$ we set
\begin{equation*}
\begin{split}
P^{(n)} &= \{ w \in \Wbb: \sigma^n w \in P(\Scal) \},\\
F^n &= \pi(P^{(n)}).
\end{split}
\end{equation*}
Any set of the form $\psi_w(F)$, $w \in \Wbb_n$, we call an \textit{$n$-complex}. Any set of the form $\psi_w(F^0)$, 
$w \in \Wbb_n$, we call an \textit{$n$-cell}.
\end{defn}

The set  $F^0$ is the ``boundary'' of $F$ and $(F^n)_{n \geq 0}$ is an increasing sequence of subsets of $F$ 
where for $n \geq 0$,
\begin{equation*}
F^{n+1} = \bigcup_{i = 1}^N \psi_i(F^n).
\end{equation*}
It can then be easily proven, by the compactness of $F$ and the contraction property of the functions $(\psi_i)$, that
\begin{equation}\label{Vn}
F = \overline{\bigcup_{n \geq 0} F^n}.
\end{equation}

\subsection{Measures on $F$}

We first define a Bernoulli measure on $\Wbb$, and then push it forward onto $F$. Let $\theta = (\theta_1, \ldots, \theta_N)$ 
be a vector such that $\sum_{i=1}^N \theta_i = 1$ and $0 < \theta_i < 1$ for each $i$. Writing 
$\theta_w := \theta_{w_1}\theta_{w_2}\ldots \theta_{w_n}$ for $w \in \Wbb_n$, and defining random variables 
$\xi_n: \Wbb \to I$ by $\xi_n(w) = w_n$, let $\tilde{\mu}_\theta$ be the unique Borel probability measure on $\Wbb$ satisfying
\begin{equation*}
\tilde{\mu}_\theta (\{ \xi_1 = w_1,\ldots , \xi_n = w_n \}) = \theta_w
\end{equation*}
for any $n \in \Nbb$ and $w \in \Wbb_n$. 

\begin{defn}
We define a \textit{Bernoulli measure} $\mu_\theta$ on $F$ to be the pushforward (through our canonical mapping $\pi$) of 
the corresponding Bernoulli measure on $\Wbb$:
\begin{equation*}
\mu_\theta := \tilde{\mu}_\theta \circ \pi^{-1}.
\end{equation*}
\end{defn}
We also define corresponding measures on each of our approximating sets $F^n$.
\begin{defn}\label{approxmeasure}
For a fixed $\theta = (\theta_1, \ldots, \theta_N)$ such that $\sum_{i=1}^N \theta_i = 1$ and $0 < \theta_i < 1$ for each $i$, let $\mu_n$ 
be the measure on $F^n$ given by
\begin{equation*}
\mu_n(x) = (\# F^0)^{-1} \sum_{w \in \Wbb_n} \theta_w \1bb_{F^0_w}(x).
\end{equation*}
\end{defn}

Note that $\mu_n$ charges every point of $F^n$ and we have \cite[Lemma 5.29]{barlow1998}.
\begin{lem}\label{measweak}
$\mu_n$ is a probability measure on $F^n$, and $\mu_n \to \mu_\theta$ weakly as $n \to \infty$.
\end{lem}

\subsection{The renormalization map}

We first recall the concept of the trace of a Dirichlet form from \cite[Chapter 4]{barlow1998}.

\begin{defn}
Let $G$ be a set and $(\Ecal,\Dcal)$ a Dirichlet form defined on $G$ and $H \subseteq G$. Define a Dirichlet form $(\tilde{\Ecal},\tilde{\Dcal})$ on $H$ by
\begin{equation*}
\tilde{\Ecal}(h,h) = \inf\left\{ \Ecal(g,g):g \in \Dcal,\ g|_H = h \right\}.
\end{equation*}
and $\tilde{\Dcal}$ is the set of functions $h$ such that the above is finite. The form $\tilde{\Ecal}$ is called the \textit{trace} of $\Ecal$ on $H$ and is denoted by $\Tr(\Ecal | H)$.
\end{defn}

This leads naturally to the notion of \textit{effective resistance} with respect to a Dirichlet form:
\begin{defn}
Let $G$ be a set and $(\Ecal,\Dcal)$ a Dirichlet form defined on $G$. Let $H_1$ and $H_2$ be disjoint subsets of $G$. The \textit{(effective) resistance} between $H_1$ and $H_2$ is
\begin{equation*}
R_\Ecal(H_1,H_2) = \left( \inf\left\{ \Ecal(g,g): g \in \Dcal,\ g|_{H_1} = 0,\ g|_{H_2} = 1 \right\} \right)^{-1},
\end{equation*}
with the convention that $0^{-1} = +\infty$. In particular, if $H_i = \{ x_i \}$ for $x_i \in G$, $i = 1,2$, then let $R_\Ecal(x_1,x_2) = R_\Ecal(\{x_1\},\{x_2\})$. If this defines a metric on $G$ (after extending it such that $R_\Ecal(x,x) = 0$ for all $x \in G$) then we call it the \textit{resistance metric} associated with $(\Ecal,\Dcal)$.
\end{defn}

In this section we seek to define a Dirichlet form on each of the $F^n$ respectively such that the sequence of Dirichlet forms 
is ``nested'', in the sense of taking traces. From this sequence we can construct a Dirichlet form on $F$ as a limit. We follow 
closely the approach given in Barlow \cite{barlow1998}.

Let $(F,(\psi_i)_{1 \leq i \leq N})$ be a generalized p.c.f.s.s.~set with a Bernoulli measure $\mu = \mu_\theta$. Define $r = (r_1,\ldots,r_N)$ 
to be a \textit{resistance vector} of positive numbers. Each $r_i$ roughly corresponds to the ``size'' of the subset $F_i \subseteq F$. 
For $n \geq 0$ let $\Dbb_n$ be the set of conservative Dirichlet forms on $F^n$. Observe that since $F^n$ is finite, $\Dbb_n$ is in 
direct correspondence with the set of conductance matrices on $F^n$. For $\Ecal \in \Dbb_0$ we write
\[ \Ecal(f,g) = \frac12 \sum_{x,y\in F^0} a(x,y) (f(x)-f(y))(g(x)-g(y)) = -f^TAg, \]
where $A = (a(x,y))_{x,y\in F^0}$ is a matrix of conductances with the diagonal terms given by $a(x,x) =-\sum_{y\neq x} a(x,y)$.

\begin{defn}\label{operations}
We define the following maps as in \cite{barlow1998}:
\begin{enumerate}
\item The \textit{replication} operation $R: \Dbb_0 \to \Dbb_1$ is given by
\begin{equation*}
R(\Ecal)(f,g) = \sum_{i=1}^N r_i^{-1} \Ecal(f \circ \psi_i , g \circ \psi_i).
\end{equation*}
The subset $F^1$ can be viewed as $N$ copies of $F^0$. The Dirichlet form $R(\Ecal)$ is simply the sum of $\Ecal$ evaluated on each of 
these copies, weighted by $r$.
\item The \textit{trace} operation $T: \Dbb_1 \to \Dbb_0$ is given by
\begin{equation*}
T(\Ecal) = \Tr(\Ecal | F^0).
\end{equation*}
\item The \textit{renormalization map} is $\Lambda = T \circ R: \Dbb_0 \to \Dbb_0$.
\end{enumerate}
\end{defn}

\begin{rem}
Note that $\Lambda$ is positively homogeneous: if $c > 0$ then $\Lambda(c\Ecal) = c\Lambda(\Ecal)$. However it is not in general 
linear, because $T$ is non-linear.
\end{rem}

%

The replication operation $R$ can be regarded as a mapping $R: \bigcup_n \Dbb_n \to \bigcup_n \Dbb_n$ such that 
if $\Ecal \in \Dbb_n$ then $R(\Ecal) \in \Dbb_{n+1}$ with
\begin{equation*}
R(\Ecal)(f,g) = \sum_{i=1}^N r_i^{-1} \Ecal(f \circ \psi_i , g \circ \psi_i).
\end{equation*}
Notice now that if $\Ecal \in \Dbb_0$, then $R^n(\Ecal) \in \Dbb_n$ with
\begin{equation*}
R^n(\Ecal)(f,g) = \sum_{w \in \Wbb_n} r_w^{-1} \Ecal(f \circ \psi_w , g \circ \psi_w).
\end{equation*}
where $r_w := r_{w_1} \ldots r_{w_n}$.

The fixed point problem is to find eigenvectors of $\Lambda$, that is, Dirichlet forms $\Ecal \in \Dbb_0$ such that there exists $\rho > 0$ with
\begin{equation*}
\Lambda(\Ecal) = \rho^{-1} \Ecal.
\end{equation*}
For such a form, let $\Ecal^{(0)} := \Ecal \in \Dbb_0$ and for $n \geq 1$ put $\Ecal^{(n)} := \rho^n R^n(\Ecal) \in \Dbb_n$. Thus we have 
a nested sequence: if $m \leq n$ then
\begin{equation*}
\Tr(\Ecal^{(n)}|F^m) = \Ecal^{(m)}.
\end{equation*}

\begin{defn}
Let $\Ecal \in \Dbb_0$ be an eigenvector of $\Lambda$ as above. If $0 < r_i\rho^{-1} < 1$ for all $1 \leq i \leq N$, then we say that 
$(\Ecal,r)$ is a \textit{regular fixed point}. If $\Ecal$ is irreducible (see \cite[Definition 4.3]{barlow1998}), then we call $\Ecal$ a \textit{non-degenerate fixed point} of $\Lambda$.
\end{defn}

The term ``non-degenerate'' corresponds to the irreducibility of the Markov process associated with $\Ecal$. If we were to take a 
degenerate fixed point and construct Dirichlet forms $\Ecal^{(n)}$ as before, then all of our associated Markov processes would 
be restricted to only a small part of the fractal. We thus restrict our attention to non-degenerate fixed points.

\begin{exmp}\label{Sierpexmp}
Let $F$ be the Sierpinski gasket. $F^0 \subseteq F$ is then a set of three points, the outermost vertices of the gasket. Labelling 
these vertices $F^0 =: \{ 1,2,3 \}$, $\Lambda$ has a non-degenerate fixed point $\Ecal_A$ where $a_{ij} = 1$ for 
$i \neq j$, $i,j \in F^0$ and $\rho = \frac{5}{3}$. If instead we let $a_{12} = 1$, $a_{23} = 0$, $a_{31} = 0$ then we see that 
$\Ecal_A$ is again a fixed point, this time degenerate, with $\rho = 2$.
\end{exmp}

\subsection{Fixed-point diffusions}\label{fpd}

We can now define a Dirichlet form on $F$, closely following the approach given in \cite{Kigami2001}. 
Let $\Scal = (F,(\psi_i)_{1 \leq i \leq N})$ be a (connected) generalized p.c.f.s.s.~set and $r$ a resistance vector. Let $\mu = \mu_\theta$ 
be a Bernoulli measure on $F$. Let $r_{\min} = \min_ir_i$ and $r_{\max} = \max_ir_i$, and let $\theta_{\min}$ and $\theta_{\max}$ be 
defined likewise. For each $n \geq 0$ let $\mu_n$ be the measure on $F^n$ given by Definition \ref{approxmeasure}. Suppose the
renormalization map $\Lambda$ has a non-degenerate regular fixed point $\Ecal^{(0)} \in \Dbb_0$ with eigenvalue 
$\rho^{-1}$, $\rho > 0$. Construct the nested sequence of Dirichlet forms $(\Ecal^{(n)})_n$ as
\begin{equation*}
\Ecal^{(n)}(f,g) := \rho^n\sum_{w \in \Wbb_n} r_w^{-1} \Ecal^{(0)}(f \circ \psi_w , g \circ \psi_w), \quad f,g \in C(F^n).
\end{equation*}
For each $n \geq 0$, let $X^n = (X^n_t)_{t \geq 0}$ be the Markov process associated with $\Ecal^{(n)}$ on $\Lcal^2(F^n,\mu_n)$.
From now on we identify real valued functions on $F$ with their restrictions to $F^n$ to simplify notation.

Observe that if $f$ is a real-valued function on $F$, then the sequence $(\Ecal^{(n)}(f,f))_n$ is non-decreasing by the properties of 
the trace operator and we can define our limiting form
\begin{equation*}
\begin{split}
&D = \left\{ f \in C(F): \sup_{n} \Ecal^{(n)}(f,f) < \infty \right\},\\
&\Ecal(f,f) = \sup_{n} \Ecal^{(n)}(f,f),\quad f \in D,
\end{split}
\end{equation*}
where $C(F)$ is the space of real-valued continuous functions on $F$.

\begin{thm}\label{fpDform}
The pair $(\Ecal,D)$ is a regular local irreducible Dirichlet form on $\Lcal^2(F,\mu)$. It has an associated resistance 
metric generating a topology that is equivalent to the existing topology on $F$.
\end{thm}

\begin{proof}
The proof is identical to proofs of similar results in \cite{Kigami2001}. The resistance metric result is from Theorem~3.3.4. $(\Ecal,D)$ is a 
local regular Dirichlet form by Theorem~3.4.6. It is irreducible since each $\Ecal^{(n)}$ is irreducible and the union of the $F^n$ is dense 
in $F$. It is easy to verify that the form is densely defined given regularity: $D$ is dense in $C(F)$ in the uniform norm. 
Since $\mu(F) < \infty$, $D$ also is 
dense in $C(F)$ in the $\Lcal^2$ norm. And $C(F)$ is dense in $\Lcal^2(F,\mu)$ by the compactness of $F$.
\end{proof}

By \cite{Fukushima2010} there therefore exists a $\mu$-symmetric diffusion $X = (X_t)_{t \geq 0}$ on $F$ associated with $(\Ecal,D)$ 
and $\Lcal^2(F,\mu)$. Let $\Dcal_F[0,\infty)$ be the space of c\`{a}dl\`{a}g functions $f: [0,\infty) \to F$. Then the following also holds:
\begin{thm}\label{fpcvgce}
$X^n \to X$ weakly in $\Dcal_F[0,\infty)$. Precisely, if $(x_n)_n$ is a sequence in $F$ such that $x_n \in F^n$ for each $n$ and $x_n \to x \in F$, 
then
\begin{equation*}
(X^n, \Pbb^{x_n}) \to (X, \Pbb^x)
\end{equation*}
weakly in $\Dcal_F[0,\infty)$.
\end{thm}
\begin{proof}
By the fact that the sequence $(\Ecal^{(n)})_n$ of Dirichlet forms are compatible, their induced resistance metrics must agree. Therefore if each $F^n$ and $F$ are interpreted as metric spaces equipped with their respective resistance metrics (using Theorem \ref{fpDform}), we see that each $F^n$ is isometrically embedded in $F$. In particular, by Theorem \ref{fpDform} the resistance metric on $F$ induces the existing topology on $F$. Since $F$ is compact and $\bigcup_{n \geq 0} F^n$ is dense in $F$, the increasing sequence of finite subsets $F^n$ must converge to $F$ in the Hausdorff topology on compact subsets of $F$ equipped with its resistance metric. Taking into account that we also have the weak convergence of Lemma \ref{measweak}, the result follows directly from \cite[Theorem 7.1]{croydon2016a}, since all of the spaces $F^n$ and $F$ are compact (see \cite[Remark 1.3(b)]{croydon2016a}).
\end{proof}

\section{Non-fixed-point diffusions}\label{nfpd}

The advantage of using a fixed point $\Ecal \in \Dbb_0$ of the renormalization map $\Lambda$ 
is that it is easy to produce a nested sequence of Dirichlet forms. We will see in this section that it is in fact possible to consider 
forms $\Ecal \in \Dbb_0$ that are \textit{not} fixed points, but where $\Lambda(\Ecal)$ is sufficiently easy to understand. 
Now let $\Scal = (F,(\psi_i)_{1 \leq i \leq N})$ be a connected generalized p.c.f.s.s.~set and $r$ a resistance vector. Let $\mu = \mu_\theta$ 
be a Bernoulli measure on $F$. For each $n \geq 0$ let $\mu_n$ be the measure on $F^n$ given by Definition \ref{approxmeasure}. 
Let $r_{\min}, r_{\max}, \theta_{\min}, \theta_{\max}$ be defined as before.

\begin{ass}\label{1paramass}
There exists a one-parameter family $\Dbb \subseteq \Dbb_0$ of forms such that if $\Ecal \in \Dbb$, then $\rho_\Ecal\Lambda(\Ecal) \in \Dbb$ 
for some $\rho_\Ecal > 0$. It is such that there exists a parametrisation $\Dbb = (\Ecal^{(0)}_v)_{v > 0}$ where $\Ecal^{(0)}_v$ only has 
edges of conductance $1$ or $v$, and there exists at least one edge of conductance $v$. Call $\Dbb$ a \textit{one-parameter invariant family}
with respect to $\Lambda$. Additionally, we assume that the family is \textit{asymptotically regular}, that is, there exists $u \in [0,\infty)$ 
such that
\begin{equation*}
r_{\max}\sup_{v  > u} \rho_{\Ecal^{(0)}_v}^{-1}  <1.
\end{equation*}
\end{ass}

\begin{rem}\label{rem:vicsek}
We specified edges to have conductance 1 or $v$; the constant 1 here is arbitrary, due to the positive homogeneity of $\Lambda$.
Other examples can be found in \cite{hambly1998}.
\end{rem}


In addition to the above assumption, we need another technical assumption on the structure of $F$, in particular its diameter in the 
resistance metric. For $v > 0$ let $R_v$ be the resistance metric on $F^0$ associated with $\Ecal^{(0)}_v$. Let $\diam(F^0,R_v)$ 
be the diameter of $F^0$ with respect to $R_v$. We require that $F^0$ does not ``shrink to zero'' in $R_v$ as $v \to \infty$. This 
can be verified geometrically. Note that in the following result we interpret $(F^0,\Ecal^{(0)}_v)$ as a graph with vertex set $F^0$ and edge set containing $(x,y)$, $x \neq y$ if and only if the $\Ecal^{(0)}_v$-conductance between $x$ and $y$ is strictly positive.

\begin{lem}\label{posdiam}
Suppose $\Dbb = (\Ecal^{(0)}_v)_{v > 0}$ is an asymptotically regular one-parameter invariant family with respect to $\Lambda$. 
The following are equivalent:
\begin{enumerate}
\item $\inf_{v > 0} \diam(F^0,R_v) > 0$.
\item The subgraph of $(F^0,\Ecal^{(0)}_v)$ generated by removing all of the edges of conductance $1$ is disconnected.
\end{enumerate}
\end{lem}

\begin{proof}
Suppose (2) holds. Evidently $(F^0,\Ecal^{(0)}_v)$ must have at least one edge of conductance $1$. On the generated subgraph there 
exists by assumption $x,y \in F^0$ in distinct connected components. Consider a function $f$ that takes the value $1$ on the connected 
component containing $x$ and takes the value $0$ elsewhere. Then for every $v > 0$,
\begin{equation*}
R_v(x,y) \geq \frac{(f(x) - f(y))^2}{\Ecal^{(0)}_v(f,f)} \geq \frac{r_{\min}}{\# \text{edges of conductance 1 in $(F^0,\Ecal^{(0)}_v)$} }
\end{equation*}
which does not depend on $v$ and is positive. So $\inf_{v > 0} \diam(F^0,R_v) > 0$.

Now suppose (2) does not hold. Then for any $x,y \in F^0$ there is a path between $x$ and $y$ in $F$ consisting only of edges of 
conductance $v$. Thus
\begin{equation*}
R_v(x,y) \leq v^{-1} r_{\max} \cdot \# \text{edges of conductance $v$ in $(F^0,\Ecal^{(0)}_v)$},
\end{equation*}
which only depends on $v$ through the term $v^{-1}$, as before. Thus $R_v(x,y) \to 0$ as $v \to \infty$. There are only a finite number of 
pairs $x,y \in F^0$, so $\diam(F^0,R_v) \to 0$ as $v \to \infty$ and so $\inf_{v > 0} \diam(F,R_v) = 0$.
\end{proof}

The first property in the above lemma is the important one; the second is just a simple geometric way of verifying it.

\begin{defn}
 We will say that a one-parameter invariant family $\Dbb$ with the first property in the above lemma is \textit{non-vanishing}.
\end{defn}

\begin{defn}
Define functions $\rho: (0,\infty) \to (0,\infty)$ and $\alpha: (0,\infty) \to (0,\infty)$ such that
\begin{equation*}
\Lambda(\Ecal^{(0)}_v) = \rho(v)^{-1}\Ecal^{(0)}_{\alpha(v)}
\end{equation*}
for any $v > 0$.
\end{defn}

The functions $\rho$ and $\alpha$ can be computed by the standard method of taking traces of conductance matrices, see 
\cite[Chapter 4]{barlow1998}. It follows that $\rho$ and $\alpha$ are rational functions of finite-degree polynomials. 
The non-vanishing property allows us to get some additional control on these functions.

\begin{prop}\label{CGexist}
Suppose $\Dbb = (\Ecal^{(0)}_v)_{v > 0}$ is an asymptotically regular non-vanishing one-parameter invariant family with respect to $\Lambda$. 
Then there exists a finite positive limit
\begin{equation*}
\rho_G := \lim_{v \to \infty} \rho(v)
\end{equation*}
such that $\rho_G > r_{\max}$.
\end{prop}

\begin{proof}
$\rho$ is a positive rational function so as $v \to \infty$ we have $\rho(v) \to \rho_G$ for some $\rho_G \in [0,\infty]$. Recall the replication 
and trace maps from Definition \ref{operations}. Consider $\rho(v) R(\Ecal^{(0)}_v) \in \Dbb_1$. By definition, the trace of this form on 
$F^0$ is $\Ecal^{(0)}_{\alpha(v)} \in \Dbb_0$. If $\rho(v) \to \infty$ as $v \to \infty$, then the diameter of $F^1$ with respect to the resistance metric induced by $\rho(v) R(\Ecal^{(0)}_v)$ must tend to $0$ as $v \to \infty$. Then we would have $\diam(F^0,R_{\alpha(v)}) \to 0$ 
which contradicts the non-vanishing property. Finally $\rho_G > r_{\max}$ by the asymptotic regularity property in Assumption \ref{1paramass}.
\end{proof}

\begin{prop}\label{alphalim}
Suppose $\Dbb = (\Ecal^{(0)}_v)_{v > 0}$ is an asymptotically regular non-vanishing one-parameter invariant family with respect 
to $\Lambda$. The following are equivalent:
\begin{enumerate}
\item $\lim_{v \to \infty} \alpha(v) = \infty$.
\item For all $x,y \in F^0$, there exists a path in $(F^0,\Ecal^{(0)}_v)$ from $x$ to $y$ consisting only of edges of conductance $v$ 
if and only if there is a path in $(F^1,R(\Ecal^{(0)}_v))$ from $x$ to $y$ consisting only of edges of conductance of the form $r_i^{-1}v$.
\item There exist $x,y \in F^0$ such that the $\Ecal^{(0)}_v$-conductance between $x$ and $y$ is $v$ and there is a path in 
$(F^1,R(\Ecal^{(0)}_v))$ from $x$ to $y$ consisting only of edges of conductance of the form $r_i^{-1}v$.
\end{enumerate}
Moreover, if the above statements hold then there exists $\beta > 0$ such that
\begin{equation*}
\lim_{v \to \infty} \frac{\alpha(v)}{v} = \frac{1}{\beta}.
\end{equation*}
\end{prop}

\begin{proof}
Obviously (2) implies (3) under Assumption \ref{1paramass}.

Suppose (3) holds. Consider $\rho(v) R(\Ecal^{(0)}_v) \in \Dbb_1$. By definition, the trace of this form on $F^0$ is $\Ecal^{(0)}_{\alpha(v)} 
\in \Dbb_0$. Since there is a path in $(F^1,\rho(v)R(\Ecal^{(0)}_v))$ from $x$ to $y$ consisting only of edges of conductance of the form 
$r_i^{-1}v\rho(v)$, and $\lim_{v \to \infty}v\rho(v) = \infty$ by the previous proposition, it follows that $R_{\alpha(v)}(x,y) \to 0$ as 
$v \to \infty$. Hence we must have $\lim_{v \to \infty}\alpha(v) = \infty$ and so (1) holds.

Now suppose (1) holds. Let $x,y \in F^0$ such that there exists a path in $(F^0,\Ecal^{(0)}_v)$ from $x$ to $y$ consisting only of edges 
of conductance $v$. By the assumption of (1) this implies that
\begin{equation*}
\lim_{v \to \infty}R_{\alpha(v)}(x,y) = \lim_{v \to \infty}R_{v}(x,y) = 0.
\end{equation*}
Considering $\rho(v) R(\Ecal^{(0)}_v) \in \Dbb_1$ again, 
we see that this Dirichlet form only has edges of conductance $r_i^{-1}\rho(v)$ or $r_i^{-1}v\rho(v)$, for $1 \leq i \leq N$. We know from 
the previous proposition that as $v \to \infty$, $\rho(v) \to \rho_G$ and so $v\rho(v) \to \infty$. Then $R_{\alpha(v)}(x,y) \to 0$ as 
$v \to \infty$ implies that there is a path in $(F^1,\rho(v)R(\Ecal^{(0)}_v))$ from $x$ to $y$ consisting only of edges of conductance 
of the form $r_i^{-1}v\rho(v)$. Both of these implications are in fact equivalences, and hence (2) is proven.

Finally, suppose (1), (2), (3) above hold. For $v > 0$ consider the Dirichlet form $v^{-1}\rho(v)R(\Ecal^{(0)}_v) \in \Dbb_1$, which only 
has edges of conductance $r_i^{-1}v^{-1}\rho(v)$ or $r_i^{-1}\rho(v)$, for $1 \leq i \leq N$. Its trace on $F^0$ is 
$v^{-1}\Ecal^{(0)}_{\alpha(v)} \in \Dbb_0$, which only has edges of conductance $v^{-1}$ or $v^{-1}\alpha(v)$. Take $x,y \in F^0$ 
from (3). We will consider the distance $d_v$ between $x$ and $y$ in the resistance metrics associated with these two Dirichlet forms, 
which are equal since the trace operator preserves the resistance metric. First we consider $d_v$ with respect to 
$v^{-1}\rho(v)R(\Ecal^{(0)}_v) \in \Dbb_1$. We know that $\rho(v) \to \rho_G$ as $v \to \infty$, so $v^{-1}\rho(v) \to 0$. By (3) there 
is a path in $(F^1,v^{-1}\rho(v)R(\Ecal^{(0)}_v))$ from $x$ to $y$ consisting only of edges of conductance of the form $r_i^{-1}\rho(v)$, so
\begin{equation*}
d_v = O(\rho(v)^{-1}) = O(\rho_G^{-1}) = O(1) \text{ as }v \to \infty.
\end{equation*}
Now we consider $d_v$ with respect to $v^{-1}\Ecal^{(0)}_{\alpha(v)} \in \Dbb_0$. Obviously $v^{-1} \to 0$ as $v \to \infty$, and by 
our choice of $x,y \in F^0$ the $v^{-1}\Ecal^{(0)}_{\alpha(v)}$-conductance between $x$ and $y$ is $v^{-1}\alpha(v)$. So similarly 
we see that
\begin{equation*}
d_v = O((v^{-1}\alpha(v))^{-1}) = O(v \alpha(v)^{-1}) \text{ as }v \to \infty.
\end{equation*}
Combining these it follows that
\begin{equation*}
\frac{\alpha(v)}{v} = O(1) \text{ as }v \to \infty,
\end{equation*}
and so since $\alpha$ is a positive rational function, it must be that
\begin{equation*}
\alpha(v) = \frac{p(v)}{\beta q(v)}
\end{equation*}
where $p$ and $q$ are monic polynomials with $\deg p = \deg q +1$, and $\beta > 0$. So
\begin{equation*}
\lim_{v \to \infty} \frac{\alpha(v)}{v} = \frac{1}{\beta}.
\end{equation*}
\end{proof}

\begin{ass}\label{alphalimass}
The statements of Proposition~\ref{alphalim} hold with $\beta > 1$. 
\end{ass}

An asymptotically regular non-vanishing one-parameter invariant family satisfying this assumption will be known as a 
one-parameter \textit{iterable} family.

For such a family, it follows from Assumptions~\ref{1paramass} and~\ref{alphalimass} that if we define
\begin{equation*}
v_{\min} = \max \left( \left\{ v: \frac{d\alpha}{dv}(v) = 0 \right\} \cup \{ v: \alpha(v) = v \} \cup \{ v: \rho(v) = r_{\max} \} \cup \{ 0 \} \right),
\end{equation*}
then $v_{\min} \in [0,\infty)$ and $\alpha$ is strictly increasing on $(v_{\min},\infty)$. The inverse $\alpha^{-1}$ of $\alpha$ can then 
be uniquely defined on $(v_{\min},\infty)$. It is continuous, strictly increasing, satisfies $\alpha^{-1}(v) > v$ for all $v \in (v_{\min},\infty)$, 
and
\begin{equation*}
\lim_{v \to \infty} \frac{\alpha^{-1}(v)}{v} = \beta > 1.
\end{equation*}
The Assumption~\ref{alphalimass} ensures that all the iterates $\alpha^{-n} := \alpha^{-1} \circ \alpha^{-(n-1)}$ can be defined 
inductively on the same domain, that is, $\alpha^{-n}: (v_{\min},\infty) \to (v_{\min},\infty)$ for all $n \in \Nbb$. Hence 
the family is ``iterable''.

\begin{rem}
In many cases $v_{\min}$ is the largest (finite) fixed point of $\alpha$ and so $\Ecal^{(0)}_{v_{\min}}$ is a non-degenerate fixed point 
of $\Lambda$. For example, in Figure \ref{fig:gasblob} we may take $v_{\min}= 1$.
\end{rem}

Henceforth we will be operating under the assumption of iterability of the one-parameter family $\Dbb$, as only under this assumption 
do the definitions of $v_{\min}$ and $\alpha^{-n}$ make sense.

\begin{defn}
For $n \geq 0$, define $\rho_n:(v_{\min},\infty) \to (v_{\min},\infty)$ by
\begin{equation*}
\rho_n(v) = \prod_{i=1}^n \rho(\alpha^{-i}(v)).
\end{equation*}
\end{defn}

Now for each $n \geq 1$ define
\begin{equation}\label{Ecaln}
\Ecal^{(n)}_v = \rho_n(v)R^n(\Ecal^{(0)}_{\alpha^{-n}(v)}).
\end{equation}
It is simple to show, as in the previous section, that $(\Ecal^{(n)}_v)_n$ is a nested sequence of Dirichlet forms on each of the $F^n$ 
respectively. As before, we define:
\begin{defn}
Let
\begin{equation*}
\begin{split}
&D_v = \left\{ f \in C(F): \sup_{n} \Ecal^{(n)}_v(f,f) < \infty \right\},\\
&\Ecal_v(f,f) = \sup_{n} \Ecal^{(n)}_v(f,f),\quad f \in D.
\end{split}
\end{equation*}
\end{defn}

\begin{thm}\label{nfpDform}
For $v \in (v_{\min}, \infty)$, the pair $(\Ecal_v,D_v)$ is a regular local irreducible Dirichlet form on $\Lcal^2(F,\mu)$. It has an associated 
resistance metric that is equivalent to the existing topology on $F$.
\end{thm}

\begin{proof}
Again the resistance metric result comes directly from \cite[Theorem 3.3.4]{Kigami2001}. By \cite[Theorem 4.1]{Kigami1995} 
$(\Ecal_v,D_v)$ is a Dirichlet form on $\Lcal^2(F,\mu)$. For regularity and locality we can follow the respective arguments in 
\cite[Theorem 3.4.6]{Kigami2001}. Finally, the proof that $(\Ecal_v,D_v)$ is irreducible and densely defined is identical to the 
respective proofs in Theorem \ref{fpDform}.
\end{proof}

Thus we have constructed a one-parameter family of diffusions 
on $F$ which are not necessarily associated with a fixed point of the renormalization map, as \cite[Theorem 7.2.1]{Fukushima2010} 
gives us 
\begin{cor}
For each $v \in (v_{\min},\infty)$ there exists a $\mu$-symmetric diffusion $X^v = (X^v_t)_{t \geq 0}$ on $F$ with Dirichlet form 
$(\Ecal_v,D_v)$ on $\Lcal^2(F,\mu)$.
\end{cor}

Now we fix a $v \in (v_{\min},\infty)$. For each $n \geq 0$, let $X^{v,n} = (X^{v,n}_t)_{t \geq 0}$ be the continuous time random 
walk associated with $\Ecal^{(n)}_v$ on $\Lcal^2(F^n,\mu_n)$.

\begin{thm}\label{nfpcvgce}
$X^{v,n} \to X^v$ weakly in $\Dcal_F[0,\infty)$. Precisely, if $(x_n)_n$ is a sequence in $F$ such that $x_n \in F^n$ for each $n$ and 
$x_n \to x \in F$, then
\begin{equation*}
(X^{v,n}, \Pbb^{x_n}) \to (X^v, \Pbb^x)
\end{equation*}
weakly in $\Dcal_F[0,\infty)$.
\end{thm}

\begin{proof}
This follows by an identical argument to Theorem \ref{fpcvgce}.
\end{proof}

Although using a diffferent parameterization, this diffusion corresponds to that constructed in \cite{hattori1994}, \cite{hambly1998} for 
the examples considered in those papers.

\section{Construction of a limiting self-similar set}

We continue with the set-up of Section \ref{nfpd}. We have $\Scal = (F,(\psi_i)_{1 \leq i \leq N})$ a connected generalized p.c.f.s.s.~set with a one-parameter iterable family $\Dbb$ (Assumption \ref{1paramass}, Assumption \ref{alphalimass}). We aim to investigate the limit as 
$v\to\infty$ of our family $\Dbb$.
For $v > v_{\min}$ and $n \geq 0$, let $G^{v,n} = (F^n,\Ecal^{(n)}_v)$ and $\Gcal^v = (F,\Ecal_v)$. 
Let $A^{v,n} = (a^{v,n}_{xy})_{x,y \in F^n}$ be the conductance matrix associated with $G^{v,n}$.

Let $d_{GH}$ be the \textit{Gromov-Hausdorff distance} on the space of metric spaces, as defined in \cite[Definition 7.3.10]{burago2001}. 
This is in fact a metric on the space of isometry classes of non-empty compact metric spaces (\cite[Theorem 7.3.30]{burago2001}).
When working with the Gromov-Hausdorff distance on non-empty compact metric spaces, we will always identify a space with its isometry 
class. We would like to find conditions under which, taking $v \to \infty$, the space $\Gcal^v$ converges in the Gromov-Hausdorff topology 
to a (possibly different) generalized p.c.f.s.s.~set. This will allow us to understand the limiting behaviour of $X^v$. In this section we will 
prove the following:
\begin{enumerate}
\item For each $n \geq 0$, $G^{v,n}$ (equipped with its resistance metric) has a limit $H^n$ in the Gromov-Hausdorff topology as 
$v \to \infty$ (Corollary \ref{finGH}). The metric on $H^n$ is the resistance metric of some conductance matrix.
\item Given a geometric assumption on $F$ and $\Dbb$ (Assumption \ref{similass}), there exists a generalized p.c.f.s.s.~set 
$\Scal_* = (F_*, (\phi_i)_{1 \leq i \leq N})$ such that its sequence of approximating networks (as in \eqref{Vn}) is $(H^n)_{n \geq 0}$ 
(Theorem \ref{S*}).
\item The Dirichlet form associated with $H^0$ is a fixed point of the renormalization map of $\Scal_*$. We can then define a limiting 
Dirichlet form and a diffusion on $F_*$ (Theorem \ref{F*Dform}).
\end{enumerate}
In fact it will eventually be proven that if $F_*$ is equipped with the resistance metric associated with its Dirichlet form, then this metric 
space is the Gromov-Hausdorff limit as $v \to \infty$ of $\Gcal^v$.

\subsection{Preliminary results}

We first need some more precise uniform estimates on the maps $\alpha(v)$ and $\rho(v)$ as well as control on the structure of $\Gcal^v$.
\begin{lem}\label{alphabounds}
Fix a $v_0 > v_{\min}$. Then there exist positive constants $k_\alpha,K_\alpha$ dependent only on $v_0$ such that for all $m,i \geq 0$ 
and all $v \geq v_0$,
\begin{equation*}
k_\alpha \leq \frac{\alpha^{-m-i}(v)}{\alpha^{-m}(v) \beta^i} \leq K_\alpha,
\end{equation*}
\end{lem}

\begin{proof}
We will first prove that, given $v \geq v_0$, there exist $k'_\alpha(v), K'_\alpha(v) > 0$ such that
\begin{equation*}
k'_\alpha(v) \leq \frac{\alpha^{-m}(v)}{v \beta^m} \leq K'_\alpha(v)
\end{equation*}
for all $m \geq 0$. Recall that by Assumption \ref{alphalimass} we have that
\begin{equation*}
\lim_{v \to \infty} \frac{\alpha(v)}{v} = \frac{1}{\beta} < 1.
\end{equation*}
Along with the fact that $\alpha$ is a positive rational function of finite-degree polynomials, this implies that there exist $l_1, l_2 > 0$ such that
\begin{equation}\label{v-1bd}
1 - \frac{l_1}{v} \leq \frac{\alpha(v) \beta}{v} \leq 1 + \frac{l_2}{v}
\end{equation}
for all $v > v_{\min}$. Since the above lower bound may be negative, we note that we also have constant bounds
\begin{equation*}
0 < l_3 \leq \frac{\alpha(v)\beta}{v} \leq l_4
\end{equation*}
for all $v > v_{\min}$ and note that $l_3 < 1$. By the fact that $\alpha^{-1}$ is continuous, $\alpha^{-1}(v) > v$ for all 
$v \in [v_0,\infty)$ and $\frac{\alpha^{-1}(v)}{v} \to \beta > 1$ as $v \to \infty$, there exists $l_5$ such that
\begin{equation*}
1 < l_5 < \frac{\alpha^{-1}(v)}{v}
\end{equation*}
for all $v \in [v_0,\infty)$. Now write
\begin{equation*}
\begin{split}
\frac{\alpha^{-m}(v)}{v \beta^m} &= \prod_{j=1}^m \left( \frac{\alpha^{-j}(v)}{\alpha^{-(j-1)}(v) \beta} \right)\\
&= \prod_{j=1}^m \left( \frac{\alpha(\alpha^{-j}(v)) \beta}{\alpha^{-j}(v)} \right)^{-1}.
\end{split}
\end{equation*}
We start with the upper bound $K'_\alpha(v)$. We see that
\begin{equation*}
\begin{split}
\left( \frac{\alpha^{-m}(v)}{v \beta^m} \right)^{-1} &\geq \prod_{j=1}^m \max \left\{ 1 - \frac{l_1}{\alpha^{-j}(v)}, l_3 \right\}\\
&\geq \prod_{j=1}^m \max \left\{ 1 - \frac{l_1}{vl_5^j}, l_3 \right\}\\
&\geq \prod_{j=1}^m \max \left\{ 1 - \frac{l_1}{v_{\min}l_5^j}, l_3 \right\}
\end{split}
\end{equation*}
for $v \in [v_0,\infty)$. It suffices to prove that the right-hand side has a strictly positive lower bound over $m \geq 0$. Since for large $j$ 
we will have $1 - \frac{l_1}{v_{\min}l_5^j} > l_3$, it suffices to consider the asymptotics of
\begin{equation*}
\prod_{j=k}^m \left( 1 - \frac{l_1}{v_{\min}l_5^j} \right),
\end{equation*}
where $k$ is high enough such that $\frac{l_1}{v_{\min}l_5^k} < \frac{1}{2}$, say. Note that this expression is decreasing in $m$. 
Taking logarithms and noting that $\log(1+x) \sim x$ for small $x$, there is a constant $l_6 > 0$ such that
\begin{equation*}
\begin{split}
\log \prod_{j=k}^m \left( 1 - \frac{l_1}{v_{\min}l_5^j} \right) &= \sum_{j=k}^m \log \left( 1 - \frac{l_1}{v_{\min}l_5^j} \right)\\
&\geq -l_6 \sum_{j=k}^m \frac{l_1}{v_{\min}l_5^j}\\
&\geq - \frac{l_1 l_5 l_6}{v_{\min}(l_5-1)}\\
&> -\infty.
\end{split}
\end{equation*}
It follows that $\left( \frac{\alpha^{-m}(v)}{v \beta^m} \right)^{-1}$ is bounded away from zero uniformly over all $m$ and so there exists some
$K'_\alpha(v)$ bounding $\frac{\alpha^{-m}(v)}{v \beta^m}$ from above for all $m$. For the lower bound, we have that
\begin{equation*}
\begin{split}
\left( \frac{\alpha^{-m}(v)}{v \beta^m} \right)^{-1} &\leq \prod_{j=1}^m \left( 1 + \frac{l_2}{\alpha^{-j}(v)} \right)\\
&\leq \prod_{j=1}^m \exp \left( \frac{l_2}{\alpha^{-j}(v)} \right)\\
&\leq \prod_{j=1}^m \exp \left( \frac{l_2}{vl_5^j} \right)\\
&\leq \exp \left(\frac{l_2l_5}{v(l_5-1)} \right),
\end{split}
\end{equation*}
and so set
\begin{equation*}
k'_\alpha(v) = \exp \left(\frac{-l_2l_5}{v(l_5-1)} \right).
\end{equation*}

Now we can prove the result. Let us prove $k_\alpha \leq \frac{\alpha^{-m-i}(v)}{\alpha^{-m}(v) \beta^i}$ first. Setting
\begin{equation*}
k''_\alpha(v) = \frac{k'_\alpha(v)}{K'_\alpha(v)},
\end{equation*}
\begin{equation*}
K''_\alpha(v) = \frac{K'_\alpha(v)}{k'_\alpha(v)},
\end{equation*}
it is clear that
\begin{equation*}
k''_\alpha(v) \leq \frac{\alpha^{-m-i}(v)}{\alpha^{-m}(v) \beta^i} \leq K''_\alpha(v)
\end{equation*}
for all $v \geq v_0$ and $m,i \geq 0$.
For the lower bound, our job is to minimize $\frac{\alpha^{-m-i}(v)}{\alpha^{-m}(v) \beta^i}$ over $[v_0,\infty)$. Observe that it is 
enough to minimize it over the much more manageable interval $[v_0,\alpha^{-1}(v_0)]$ since any higher values of $v$ can then be 
covered by increasing the value of $m$. Thus for all $v \in [v_0,\alpha^{-1}(v_0)]$, since $\alpha^{-1}(v)$ is increasing in $v$,
\begin{equation*}
\frac{\alpha^{-m-i}(v)}{\alpha^{-m}(v) \beta^i} \geq \frac{\alpha^{-m-i}(v_0)}{\alpha^{-m}(\alpha^{-1}(v_0)) \beta^i} = 
\frac{\alpha^{-m-i}(v_0)}{\alpha^{-m-1}(v_0) \beta^i} \geq \frac{k_\alpha''(v_0)}{\beta} =: k_\alpha
\end{equation*}
and the proof is complete. The derivation of $K_\alpha$ is essentially the same.
\end{proof}

\begin{lem}\label{Cbounds}
Fix a $v_0 > v_{\min}$. There exist positive constants $k_\rho,K_\rho$ dependent only on $v_0$ such that for all $m,i \geq 0$ and all 
$v \geq v_0$,
\begin{equation*}
k_\rho \leq \frac{\rho_{m+i}(v)}{\rho_m(v) \rho_G^i} \leq K_\rho.
\end{equation*}
\end{lem}

\begin{proof}
The proof will proceed in much the same way as in the previous lemma. We first prove that there exist $k'_\rho(v), K'_\rho(v) > 0$ such that
\begin{equation*}
k'_\rho(v) \leq \frac{\rho_m(v)}{\rho_G^m} \leq K'_\rho(v)
\end{equation*}
for all $v \geq v_0$ and $m \geq 0$. We use the same constants as in the proof of Lemma~\ref{alphabounds}, and introduce 
$l_7, l_8 > 0$ satisfying
\begin{equation*}
1 - \frac{l_7}{v} \leq \frac{\rho(v)}{\rho_G} \leq 1 + \frac{l_8}{v}
\end{equation*}
for all $v > v_{\min}$, which is possible since $\rho(v)$ is a positive rational function converging to $\rho_G$ as $v \to \infty$. Also as before 
there exist $l_9,l_{10} > 0$ such that
\begin{equation*}
l_9 \leq \frac{\rho(v)}{\rho_G} \leq l_{10}
\end{equation*}
for all $v > v_{\min}$, where $l_9 < 1$.
Writing
\begin{equation*}
\frac{\rho_m(v)}{\rho_G^m} = \prod_{j=1}^m \frac{\rho(\alpha^{-j}(v))}{\rho_G},
\end{equation*}
for the lower bound we see that
\begin{equation*}
\frac{\rho_m(v)}{\rho_G^m} \geq \prod_{j=1}^m \max \left\{ 1 - \frac{l_7}{\alpha^{-j}(v)}, l_9 \right\}.
\end{equation*}
We bound the right-hand side away from zero by taking logarithms in exactly the same way as in the proof of Lemma~\ref{alphabounds}. 
For the upper bound,
\begin{equation*}
\frac{\rho_m(v)}{\rho_G^m} \leq \prod_{j=1}^m \left( 1 + \frac{l_8}{\alpha^{-j}(v)} \right)
\end{equation*}
which we bound by an exponential just as in the proof of Lemma~\ref{alphabounds}. So we have shown the existence of $k'_\rho(v), K'_\rho(v)$.

As before we now set
\begin{equation*}
k''_\rho(v) = \frac{k'_\rho(v)}{K'_\rho(v)},
\end{equation*}
\begin{equation*}
K''_\rho(v) = \frac{K'_\rho(v)}{k'_\rho(v)},
\end{equation*}
so that
\begin{equation*}
k''_\rho(v) \leq \frac{\rho_{m+i}(v)}{\rho_m(v) \rho_G^i} \leq K''_\rho(v)
\end{equation*}
for all $v \geq v_0$ and $m,i \geq 0$. From the definition we know that
\begin{equation*}
\frac{\rho_{m+i}(v)}{\rho_m(v)} = \prod_{j = 1}^i \rho(\alpha^{-(m+j)}(v)).
\end{equation*}
To do something similar to the above we need to prove some kind of monotonicity property. We will prove that the function 
$v \mapsto \alpha^{-m}(v)\rho_m(v)$ is non-decreasing for any $m$. Indeed, consider the conductance network $(F^m, \Ecal^{(m)}_v)$, 
where all edges have had their conductance multiplied by a factor of $(\alpha^{-m}(v)\rho_m(v))^{-1}$. That is, this conductance 
network has conductances of the form $r_w^{-1}$ or $r_w^{-1}\alpha^{-m}(v)^{-1}$ for $w \in \Wbb_n$. Taking the trace onto $F^0$, 
we see that the corresponding conductance network on $F^0$ has conductances of either $(\alpha^{-m}(v)\rho_m(v))^{-1}$ or 
$v(\alpha^{-m}(v)\rho_m(v))^{-1}$. Consider the monotonicity law of conductance networks: decreasing some of the edge conductances 
of the network cannot increase any of its effective conductances. On the other hand, if we were to strictly increase all of the edge conductances of a network, all of its effective conductances would necessarily strictly increase. It follows that, since $\alpha^{-m}(v)^{-1}$ is decreasing 
in $v$, $(\alpha^{-m}(v)\rho_m(v))^{-1}$ cannot be increasing in $v$. So $\alpha^{-m}(v)\rho_m(v)$ is non-decreasing in $v$.

Now as in Lemma~\ref{alphabounds}, we can similarly minimize only over $[v_0,\alpha^{-1}(v_0)]$. Using Lemma~\ref{alphabounds},
\begin{equation*}
\begin{split}
\frac{\rho_{m+i}(v)}{\rho_m(v) \rho_G^i} &= \frac{\alpha^{-m}(v)}{\alpha^{-m-i}(v)} \frac{\alpha^{-m-i}(v)\rho_{m+i}(v)}
{\alpha^{-m}(v)\rho_m(v) \rho_G^i}\\
&\geq \frac{\alpha^{-m}(v_0)}{\alpha^{-m-i}(\alpha^{-1}(v_0))}\frac{\alpha^{-m-i}(v_0)\rho_{m+i}(v_0)}
{\alpha^{-m}(\alpha^{-1}(v_0))\rho_m(\alpha^{-1}(v_0)) \rho_G^i}\\
&\geq \frac{k_\alpha \beta^{i-1}}{K_\alpha \beta^{i+1}} \frac{\rho_{m+i}(v_0)\rho(\alpha^{-1}(v_0))}{\rho_{m+1}(v_0) \rho_G^i}\\
&\geq \frac{k_\alpha \rho(\alpha^{-1}(v_0))k_\rho''(v_0)}{K_\alpha \beta^2 \rho_G}
\end{split}
\end{equation*}
and we define the right-hand side to be $k_\rho$. The derivation of $K_\rho$ is similar. 
\end{proof}

\begin{cor}\label{Cnlimit}
For any $v \geq v_0 > v_{\min}$,
\begin{equation*}
\lim_{m \to \infty} \frac{\alpha^{-m}(v)}{v\beta^m} \in [k_\alpha,K_\alpha] \subset (0,\infty),
\end{equation*}
\begin{equation}\label{eq:rholim}
\lim_{m \to \infty} \frac{\rho_m(v)}{\rho_G^m} \in [k_\rho,K_\rho] \subset (0,\infty). 
\end{equation}
\end{cor}

\begin{proof}
$\frac{v}{\alpha(v)}$ is a positive rational function so is monotone in $v$ for sufficiently large $v$. Therefore $\frac{\alpha^{-1}(v)}{v}$ 
is monotone for sufficiently large $v$. Since $\frac{\alpha^{-1}(v)}{v\beta} \to 1$ as $v \to \infty$, this implies that $\frac{\alpha^{-m}(v)}
{v\beta^m}$ is a monotone sequence in $m$ for sufficiently large $m$. By Lemma~\ref{alphabounds} this sequence is bounded in 
$[k_\alpha,K_\alpha]$ so it must converge in this interval. Similarly for the second statement, using Lemma~\ref{Cbounds}.
\end{proof}

\begin{lem}\label{diambound}
For any $v_0 > v_{\min}$,
\begin{equation*}
\sup_{v \geq v_0} \diam \Gcal^v < \infty.
\end{equation*}
\end{lem}

\begin{proof}
An adaptation of \cite[Proposition 7.10]{barlow1998} will suffice. First of all, observe that
\begin{equation*}
c := \sup_{v > v_{\min}} \diam G^{v,1} < \infty.
\end{equation*}
Fix $v \geq v_0$. Let $x \in G^{v,n}$ for some $n \geq 0$. Then there exists $w \in \Wbb_{n-1}$ and $x_1 \in G^{v,n-1}$ such that 
$x,x_1 \in F_w$. Thus
\begin{equation*}
R_v(x,x_1) \leq c r_w\rho_{n-1}(v)^{-1}.
\end{equation*}
Iterating this process, we find $y \in G^{v,0}$ such that
\begin{equation*}
R_v(x,y) \leq c \sum_{i=0}^{n-1}r_{\max}^i\rho_i(v)^{-1} \leq \frac{c\rho_G}{k_\rho(\rho_G-r_{\max})}
\end{equation*}
where we have used Lemma~\ref{Cbounds}. Thus for any $x,y \in \bigcup_{n \geq 0} F^n$,
\begin{equation*}
R_v(x,y) \leq \frac{2c\rho_G}{k_\rho(\rho_G-r_{\max})} + c.
\end{equation*}
Then $F$ is the closure of $\bigcup_{n \geq 0} F^n$ so we are done.
\end{proof}

\subsection{Finite approximations}

\begin{defn}
Fix $n \geq 0$. For $x,y \in F^n$, if $x = y$ or if there exists a path between $x$ and $y$ in $G^{v,n}$ consisting only of edges 
with conductance of the form $r_w^{-1} \rho_n(v)\alpha^{-n}(v)$, then we say that $x \sim y$. By Assumption~\ref{alphalimass}, 
$\sim$ is independent of $n$. Thus $\sim$ is an equivalence relation on $\bigcup_{n \geq 0} F^n$.
\end{defn}

As $v \to \infty$, the resistance between certain pairs $x,y$ of points in $\Gcal^v$ tends to zero. This suggests that if any 
Gromov-Hausdorff limit of the spaces $\Gcal^v$ were to exist as $v \to \infty$, then these pairs of points would have to be regarded 
as the same in that space. The limit space is thus unlikely to have the same underlying set $F$.

We start by identifying points of our original fractal for which the effective resistance decreases to zero as $v \to \infty$. For each 
$n \geq 0$ define the conductance network $(F_*^n,A^n)$ as follows: $F_*^n$ is the result of taking $F^n$ and identifying all pairs 
of points $x,y \in F^n$ for which the effective resistance between $x$ and $y$ in $G^{v,n}$ tends to zero as $v \to \infty$. 
Equivalently, $F^n_*$ is the set of equivalence classes of $F^n$ under $\sim$. The $A^n$-conductance between two distinct 
$y_1,y_2 \in F^n_*$ is defined to be
\begin{equation*}
a^n_{y_1y_2} = \sum_{\substack{x \in y_1 \\ z \in y_2}} \lim_{v \to \infty} a^{v,n}_{xz}.
\end{equation*}
It is thus always an integer multiple of $\rho_G^n$. The natural measure to use on $F^n_*$ is $\nu_n$, where for $y \in F^n_*$,
\begin{equation*}
\nu_n (\{ y \}) = \mu_n(y) = \sum_{x \in y} \mu_n(\{x\}).
\end{equation*}
That is, each class is given a mass equal to the sum of the original masses of its constituent elements. Let $\Ecal^{(n)}_*$ be the 
Dirichlet form associated with $(F_*^n,A^n)$ and let $H^n = (F_*^n,\Ecal^{(n)}_*)$. Each $\Ecal^{(n)}_*$ is clearly irreducible.

\subsection{The projection map}

The identification of points through $\sim$ induces a surjective \textit{projection map} $p: F^n \to F_*^n$ for each $n \geq 0$ which takes 
each point of $F^n$ to the class in $F_*^n$ that contains it. Recall that the equivalence relation $\sim$ is independent of $n$. This means 
that if $n > m$ then $F_*^m$ can naturally be thought of as a subset of $F_*^n$, analogously to how $F^m$ is a subset of $F^n$. Each 
element of $F_*^m$ is identified with the element of $F_*^n$ of which it is a subset. Nesting the $F_*^n$ in this way hence preserves 
the definition of the projection map $p$, so that $p$ is independent of $n$. Thus we can extend $p$ to a surjective function
\begin{equation}\label{Pintermediate}
p: \bigcup_{n \geq 0} F^n \to \bigcup_{n \geq 0} F_*^n
\end{equation}
which agrees with its previous definition when restricted to $F^n$ for any $n \geq 0$.
\begin{rem}
Each measure $\nu_n$ on $F^n_*$ is the pushforward through $p$ of the respective measure $\mu_n$ on $F^n$.
\end{rem}

We now prove a useful topological result, which is similar to results in \cite{munkres1974}:

\begin{lem}\label{topo}
Let $E_1,E_2$ be topological spaces, and $f:E_1 \times E_2 \to (-\infty,\infty]$ a continuous function (where $(-\infty,\infty]$ is the space $\Rbb$ compactified on the right). Suppose $E_2$ is compact. 
Let $g:E_1 \to (-\infty,\infty]$ be given by $g(x) = \inf_{y \in E_2} f(x,y)$. Then $g$ is a well-defined continuous function.
\end{lem}

\begin{proof}
Since $E_2$ is compact, the infimum is well-defined and moreover for each $x \in E_1$ the infimum is attained by some $y \in E_2$.
A sub-basis of the topology on $(-\infty,\infty]$ is
\begin{equation*}
\left\{ (-\infty,a):a \in (-\infty,\infty] \right\} \cup \left\{ (a,\infty]:a \in (-\infty,\infty) \right\}.
\end{equation*}
It is enough to show that $g^{-1}(S)$ is open for each set $S$ in the above sub-basis.

If $S= (-\infty,a)$, then
\begin{equation*}
\begin{split}
g^{-1}(S) &= \left\{ x \in E_1: \exists y \in E_2: f(x,y) < a \right\}\\
&= \pi_1(f^{-1}(S))
\end{split}
\end{equation*}
where $\pi_1$ is the projection map $E_1 \times E_2 \to E_1$, which is an open map. $f^{-1}(S)$ is open, thus $g^{-1}(S)$ is open.

If $S= (a,\infty]$, then since the infimum is always attained,
\begin{equation*}
g^{-1}(S) = \left\{ x \in E_1: \forall y \in E_2, f(x,y) > a \right\}.
\end{equation*}
If $g^{-1}(S)$ is empty then the proof is done, so we assume that $g^{-1}(S)$ is non-empty. Again we consider $f^{-1}(S)$. Since this is 
open, then by the definition of the product topology there exist for each $(x,y) \in f^{-1}(S)$ open sets $U_{xy} \in E_1$ and 
$V_{xy} \in E_2$ such that $(x,y) \in U_{xy} \times V_{xy} \subseteq f^{-1}(S)$. Now we fix $x_0 \in g^{-1}(S)$. This means that 
$(x_0,y) \in f^{-1}(S)$ for all $y \in E_2$, so the open sets $\{V_{x_0y}:(x_0,y) \in f^{-1}(S)\}$ cover $E_2$. $E_2$ is compact so there 
is a finite subcover $\{V_{x_0y_i}\}_i$. Now let
\begin{equation*}
U_{x_0} = \bigcap_i U_{x_0y_i},
\end{equation*}
which is a finite intersection, thus open in $E_1$. Moreover, $U_{x_0} \times V_{x_0y_i} \subseteq f^{-1}(S)$ for every $i$, so
\begin{equation*}
U_{x_0} \times E_2 = U_{x_0} \times \bigcup_i V_{x_0y_i} \subseteq f^{-1}(S).
\end{equation*}
This implies that $U_{x_0} \subseteq g^{-1}(S)$. The set $U_{x_0}$ is open and $x_0 \in U_{x_0}$. The point $x_0$ was arbitrary, 
so $g^{-1}(S)$ is open.
\end{proof}
Our analysis of the limit space will be clearer if we define the \textit{preimages} of the Dirichlet forms $\Ecal^{(n)}_*$ on $F^n$. 
Let $C(F^n)$ be the space of (continuous) real-valued functions on $F^n$. For each $n$, we define a form $\Ecal^{(n)}_\infty$ on the domain
\begin{equation*}
C_*(F^n) := \{ f \in C(F^n): x \sim y \Rightarrow f(x) = f(y) \}
\end{equation*}
such that
\begin{equation*}
\Ecal^{(n)}_\infty(f,f) = \Ecal^{(n)}_*(f \circ p^{-1}, f \circ p^{-1}).
\end{equation*}
\begin{rem}
Note that $C_*(F^n)$ is closed in $C(F^n)$, and that $f \circ p^{-1}$ makes sense here because of the definition of $C_*(F^n)$.
\end{rem}
Defining $\Ecal^{(n)}_\infty(f,f) = \infty$ for all $f \in C(F^n) \setminus C_*(F^n)$, we importantly have that for all $f \in C(F^n)$,
\begin{equation*}
\Ecal^{(n)}_\infty(f,f) = \lim_{v \to \infty} \Ecal^{(n)}_v(f,f).
\end{equation*}
With this set-up, the pairs $G^{\infty,n} := (F^n,\Ecal^{(n)}_\infty)$ and $H^n = (F_*^n,\Ecal^{(n)}_*)$ define essentially the same 
structure, and so this provides an analytical link between the networks $(G^{v,n})_{v > v_{\min}}$ and $H^n$. This allows us, for 
example, to prove the following key result:
\begin{prop}\label{Btrace}
Let $m < n$. Then the trace of $H^n$ on $F_*^m$ is $H^m$.
\end{prop}
\begin{proof}
Let $f \in C(F_*^m)$. We need to prove that
\begin{equation*}
\Ecal^{(m)}_*(f,f) = \inf \{ \Ecal^{(n)}_*(g,g): g \in C(F_*^n) ,\ g|_{F_*^m} = f \}.
\end{equation*}
Since $\Ecal^{(n)}_\infty(g,g) = \infty$ for all $g \in C(F^n) \setminus C_*(F^n)$,
\begin{equation*}
\begin{split}
\inf &\{ \Ecal^{(n)}_*(g,g): g \in C(F_*^n) ,\ g|_{F_*^m} = f \}\\
&= \inf \{ \Ecal^{(n)}_\infty(g,g): g \in C(F^n) ,\ g|_{F^m} = f \circ p \}\\
&= \inf \{ \Ecal^{(n)}_\infty(g,g): g \in C_*(F^n) ,\ g|_{F^m} = f \circ p \}.
\end{split}
\end{equation*}
Observe that we can restrict the function $g$ to only taking values between $\max_x f(x)$ and $\min_x f(x)$ inclusive, since if we let 
$g' = (g \vee \min_x f(x)) \wedge \max_x f(x)$ then $\Ecal^{(n)}_\infty(g',g') \leq \Ecal^{(n)}_\infty(g,g)$. This makes the set over 
which we are taking the infimum compact. We can thus use Lemma \ref{topo} and the fact that $\Ecal^{(n)}_v (g,g)$ is continuous in $v$:
\begin{equation*}
\begin{split}
\inf &\{ \Ecal^{(n)}_\infty(g,g): g \in C_*(F^n) ,\ g|_{F^m} = f \circ p \}\\
&= \inf \{ \lim_{v \to \infty} \Ecal^{(n)}_v(g,g): g \in C_*(F^n) ,\ g|_{F^m} = f \circ p \}\\
&= \lim_{v \to \infty} \inf \{ \Ecal^{(n)}_v(g,g): g \in C_*(F^n) ,\ g|_{F^m} = f \circ p \}\\
&= \lim_{v \to \infty} \Ecal^{(m)}_v(f \circ p,f \circ p)\\
&= \Ecal^{(m)}_\infty(f \circ p,f \circ p)\\
&= \Ecal^{(m)}_*(f,f).
\end{split}
\end{equation*}
where between lines three and four we have used the compatibility of the sequence of Dirichlet forms $(\Ecal^{(n)}_v)_n$.
\end{proof}
The above proposition immediately implies that $(\Ecal^{(n)}_*)_{n \geq 0}$ can be taken to be a nested sequence of Dirichlet forms 
on $\bigcup_{n \geq 0} F^n_*$ (by identifying a function with its restriction to $F^n_*$). It therefore induces a well-defined resistance 
metric $d$ on $\bigcup_{n \geq 0} F_*^n$. Recall that $R_v$ is the resistance metric on $\Gcal^v$. We thus have the following result:
\begin{prop}\label{metriccvgce}
Let $n \geq 0$ and $x,y \in F^n$. Then
\begin{equation*}
\lim_{v \to \infty} R_v(x,y) = d(p(x),p(y)).
\end{equation*}
\end{prop}
\begin{proof}
If $x \sim y$, then $R_v(x,y) \to 0 = d(p(x),p(y))$ as $v \to \infty$. So assume that we do not have $x \sim y$. By Lemma~\ref{topo} 
and the fact that we can take the infimum over a compact set (as in the proof of Proposition \ref{Btrace}),
\begin{equation*}
\begin{split}
\lim_{v \to \infty} R_v(x,y)^{-1} &= \lim_{v \to \infty} \inf \{ \Ecal^{(n)}_v(f,f) : f \in C(F^n),\ f(x) = 1,\ f(y) = 0 \}\\
&= \inf \{ \lim_{v \to \infty} \Ecal^{(n)}_v(f,f) : f \in C(F^n),\ f(x) = 1,\ f(y) = 0 \}\\
&= \inf \{ \Ecal^{(n)}_\infty(f,f) : f \in C_*(F^n),\ f(x) = 1,\ f(y) = 0 \}\\
&= \inf \{ \Ecal^{(n)}_*f,f) : f \in C(F_*^n),\ f(p(x)) = 1,\ f(p(y)) = 0 \}\\
&= d(p(x),p(y))^{-1}.
\end{split}
\end{equation*}
Here again we have used the fact that $\Ecal^{(n)}_\infty(g,g) = \infty$ for all $g \in C(F^n) \setminus C_*(F^n)$.
\end{proof}
We conclude this section with a small convergence result, showing the link between the resistance metrics of the spaces $G^{v,n}$ and $H^n$.
\begin{defn}[\cite{burago2001}]\label{distortion}
Let $(M_1,d_1),(M_2,d_2)$ be metric spaces. Let $f:M_1 \to M_2$ be a surjective function. The \textit{distortion} of $f$ is
\begin{equation*}
\dis f := \sup_{x_1,x_2 \in M_1} \{ |d_1(x_1,x_2) - d_2(f(x_1),f(x_2))| \}.
\end{equation*}
\end{defn}
\begin{lem}\label{GHdis}
Let $(M_1,d_1),(M_2,d_2)$ be compact metric spaces and let $f:M_1 \to M_2$ be a surjective function. Then
\begin{equation*}
d_{GH}(M_1,M_2) \leq \frac{1}{2} \dis f.
\end{equation*}
\end{lem}
\begin{proof}
Direct corollary of \cite[Theorem 7.3.25]{burago2001}.
\end{proof}
\begin{cor}\label{finGH}
For each $n \geq 0$, $G^{v,n} \to H^n$ as $v \to \infty$ in the Gromov-Hausdorff metric.
\end{cor}
\begin{proof}
All of these metric spaces have a finite number of elements, so are compact. Fix $n \geq 0$. The map $p$ is surjective, and by 
Proposition~\ref{metriccvgce} its distortion from $G^{v,n}$ to $H^n$ tends to $0$ as $v \to \infty$ (as the distortion takes a supremum over a finite number 
of pairs of elements). The result follows by Lemma~\ref{GHdis}.
\end{proof}

\subsection{The limit set}

The natural thing to do now would be to take the $F_*^n$ as an approximating sequence to some generalized p.c.f.s.s.~set. We require a 
further assumption about our set-up, in order to avoid a number of pathological examples:
\begin{ass}\label{similass}
For all $n \geq 0$, $x,y \in G^{v,n}$ and $1 \leq i \leq N$, $x \sim y$ if and only if $\psi_i(x) \sim \psi_i(y)$. This property will be 
known as \textit{$\Dbb$-injectivity} of the functions $\psi_i$ with respect to the one-parameter family $\Dbb$.
\end{ass}

\begin{rem}
In fact to verify that Assumption \ref{similass} holds it is enough to verify it in the case $n = 0$. Indeed, assume that $\Dbb$-injectivity 
does not hold. Then for some $n$ and $i$ there exists $x,y \in G^{v,n}$ with $\psi_i(x) \sim \psi_i(y)$ but not $x \sim y$ (since 
$x \sim y$ always implies $\psi_i(x) \sim \psi_i(y)$). By definition, there is a path between $\psi_i(x)$ and $\psi_i(y)$ consisting only of 
edges with conductance of the form $r_w^{-1} \rho_n(v)\alpha^{-n}(v)$, but this path cannot be contained in $F_i$ (otherwise its 
preimage in $\psi_i$ is a path that implies $x \sim y$). Thus there exist $z_1,z_2 \in G^{v,0}$ such that $\psi_i(z_1)$ and $\psi_i(z_2)$ lie 
on the path, and $x \sim z_1$ and $y \sim z_2$. So $\psi_i(z_1) \sim \psi_i(z_2)$ but not $z_1 \sim z_2$ and we are reduced to the 
case $n = 0$.
\end{rem}

Henceforth we take Assumption \ref{similass} to be true. The resistance metric $d$ on $\bigcup_{n \geq 0} F_*^n$ is bounded by 
Lemma \ref{diambound} and Proposition \ref{metriccvgce}. Let $F_*$ be the completion of $\bigcup_{n \geq 0} F_*^n$ with respect 
to $d$. Then $(F_*,d)$ is a bounded complete metric space. To define contractions $\phi_i$, recall the original generalized p.c.f.s.s.~structure 
$(F,(\psi_i)_{1 \leq i \leq N})$. For $x \in \bigcup_{n \geq 0} F_*^n$ and $1 \leq i \leq N$, the natural definition of $\phi_i$ is
\begin{equation*}
\phi_i(x) = p \circ \psi_i \circ p^{-1}(x).
\end{equation*}
This is well-defined and injective because if $x,y \in \bigcup_{n \geq 0} F^n$ then $x \sim y$ if and only if $\psi_i(x) \sim \psi_i(y)$ by 
Assumption \ref{similass}. It's easy to check that
\begin{equation}\label{F*approx}
\bigcup_{i=1}^N \phi_i \left( F^n_* \right) = F_*^{n+1}
\end{equation}
for each $n \geq 0$ and thus that
\begin{equation*}
\bigcup_{i=1}^N \phi_i \left( \bigcup_{n \geq 0} F_*^n \right) = \bigcup_{n \geq 0} F_*^n.
\end{equation*}
\begin{defn}[Harmonic extension]
Let $v > v_{\min}$ and $n \geq 0$. For a function $f: F^n \to \Rbb$, its \textit{harmonic extension} to $\Gcal^v$ is the unique continuous function $g:F \to \Rbb$ for which
\begin{equation*}
\Ecal_v(g,g) = \Ecal_v^{(n)}(f,f).
\end{equation*}
This can be shown to exist by the following argument: by the same reasoning as \cite[Lemma 2.2.2]{Kigami2001}, $f$ can be extended uniquely to a function $f_1:\bigcup_{n \geq 0} F^n \to \Rbb$ for which
\begin{equation*}
\Ecal_v^{(m)}(f_1,f_1) = \Ecal_v^{(n)}(f,f)
\end{equation*}
for all $m \geq n$. Then for all $x,y \in \bigcup_{n \geq 0} F^n$ we have that
\begin{equation*}
|f_1(x) - f_1(y)|^2 \leq R_v(x,y)\Ecal_v^{(n)}(f,f),
\end{equation*}
so there is a unique continuous extension $g$ of $f_1$ to $F$.
\end{defn}
Our definition of harmonic extension given above coincides with that given in \cite[Section 3.2]{Kigami2001}. We notice by construction that if the function $f$ is constant on $\psi_w(F^0)$ for some $w \in \Wbb_n$, then its harmonic extension to $\Gcal^v$ must be constant on $F_w$ for any $v > v_{\min}$.
\begin{prop}\label{phiinject}
For each $1 \leq i \leq N$, $\phi_i$ can be extended uniquely to an injective contraction from $F_*$ to itself with Lipschitz constant at most 
$r_{\max}\rho_G^{-1}$, and $(F_*,(\phi_i)_{1 \leq i \leq N})$ is a self-similar structure.
\end{prop}
\begin{proof}
Notice (by using \eqref{Ecaln} for example) that if $f \in C(F^n)$ for $n \geq 1$ then
\begin{equation}\label{EcalIterate}
\begin{split}
\Ecal^{(n)}_v(f,f) &= \rho(\alpha^{-n}(v)) R(\Ecal^{(n-1)}_{\alpha^{-1}(v)})(f,f)\\
&= \rho(\alpha^{-n}(v))\sum_{i=1}^N r_i^{-1}\Ecal^{(n-1)}_{\alpha^{-1}(v)}(f \circ \psi_i, f \circ \psi_i).
\end{split}
\end{equation}
for $v > v_{\min}$. It follows that if $x,y \in \bigcup_{n \geq 0} F_*^n$, then there exists $n \geq 1$ such that $x,y \in F_*^{n-1}$ and 
so $\phi_i(x), \phi_i(y) \in F_*^n$. Therefore, similarly to the proof of Proposition \ref{Btrace},
\begin{equation*}
\begin{split}
&d (\phi_i(x),\phi_i(y))^{-1} = \inf \left\{ \Ecal^{(n)}_*(f,f): f \in C(F_*^n),\ f(\phi_i(x)) = 0,\ f(\phi_i(y)) = 1 \right\}\\
&= \inf \left\{ \lim_{v \to \infty} \Ecal^{(n)}_v(f,f): f \in C_*(F^n),\ f|_{p^{-1}(\phi_i(x))} = 0,\ f|_{p^{-1}(\phi_i(y))} = 1 \right\}\\
&= \inf \left\{ \lim_{v \to \infty} \Ecal^{(n)}_v(f,f): f \in C_*(F^n),\ f|_{\psi_i(p^{-1}(x))} = 0,\ f|_{\psi_i(p^{-1}(y))} = 1 \right\}\\
&\geq \rho_G r_i^{-1} \inf \left\{ \lim_{v \to \infty} \Ecal^{(n-1)}_{\alpha^{-1}(v)}(f,f): f \in C_*(F^{n-1}),\ f|_{p^{-1}(x)} = 0,\ f|_{p^{-1}(y)} = 1 \right\}\\
&= \rho_G r_i^{-1} \inf \left\{ \Ecal^{(n-1)}_*(f,f): f \in C(F_*^{n-1}),\ f(x) = 0,\ f(y) = 1 \right\}\\
&= \rho_G r_i^{-1} d(x,y)^{-1}
\end{split}
\end{equation*}
Thus $d (\phi_i(x),\phi_i(y)) \leq r_{\max}\rho_G^{-1} d(x,y)$, and $r_{\max}\rho_G^{-1} < 1$ by Proposition \ref{CGexist}. 
So the $\phi_i$ are all contractions, and so there exists a unique continuous extension of each $\phi_i$ to the whole of $F_*$, which 
is a contraction with the same constant. By \cite[Theorem 1.1.4]{Kigami2001}, there exists a unique non-empty compact subset 
$K$ of $F_*$ satisfying
\begin{equation*}
K = \bigcup_{i=1}^N \phi_i(K).
\end{equation*}
Take $x \in K$. If $y \in \bigcup_{n \geq 0} F_*^n$ then for each $n \geq 0$ there exists $w \in \Wbb_n$ such that $y \in \phi_w(F_*)$. 
Thus $d(y,\phi_w(x)) < (r_{\max}\rho_G^{-1})^n \diam(F_*)$ and of course $\phi_w(x) \in K$, so $y$ is a limit point of $K$, so $y \in K$. 
So $\bigcup_{n \geq 0} F_*^n \subseteq K$, so by denseness we must have $K = F_*$. Hence $F_*$ is compact and
\begin{equation*}
F_* = \bigcup_{i=1}^N \phi_i(F_*).
\end{equation*}
A consequence of compactness is that for any $n \geq 0$, $w \in \Wbb$, we have
\begin{equation*}
\overline{\phi_w \left(\bigcup_{n \geq 0} F_*^n \right)} = \phi_w(F_*).
\end{equation*}

We now prove that each $\phi_i$ is injective on $F_*$. Suppose $n \geq 0$ and $w,u \in \Wbb_n$ such that $\phi_w(F_*)$ and $\phi_u(F_*)$ 
are disjoint. Thus $F_w$ and $F_u$ are disjoint and in $G^{v,n}$ the two subsets $\psi_w(F^0)$ (which is equal to $F^n \cap F_w$) and $\psi_u(F^0)$ are not connected 
by any path whose resistance tends to zero as $v \to \infty$. Using Assumption \ref{similass} we can say the same thing about $\psi_{iw}(F^0)$ 
and $\psi_{iu}(F^0)$ in $G^{v,n+1}$. For each $v > v_{\min}$, take a function $f$ on $F^{n+1}$ that is bounded in $[0,1]$, takes the value 
$1$ on $\psi_{iw}(F^0)$ and the value $0$ on $\psi_{iu}(F^0)$ and $f(x) = f(y)$ if $x \sim y$. Then consider its harmonic extension to $\Gcal^v$, also called $f$. Now $f$ is bounded in $[0,1]$, takes the value $1$ on $F_{iw}$ and the value $0$ on $F_{iu}$, so for any 
$x_1 \in F_{iw}$ and $x_2 \in F_{iu}$,
\begin{equation*}
\begin{split}
R_v(x_1,x_2) &\geq \Ecal_v(f,f)^{-1} = \Ecal^{(n+1)}_v(f,f)^{-1}\\
&\geq r_{\min}^{n+1} \left( \sum_{z_1,z_2 \in F^{n+1}: z_1 \nsim z_2} \rho_{n+1}(v)(f(z_1) - f(z_2))^2 \right)^{-1}\\
&\geq K_\rho^{-1} (r_{\min}\rho_G^{-1})^{n+1} \#\{ z_1,z_2 \in F^{n+1}: z_1 \nsim z_2 \}^{-1}\\
&=: c_{n+1} > 0
\end{split}
\end{equation*}
where we have fixed some $v_0 > v_{\min}$ such that $v \geq v_0$ and used Lemma \ref{Cbounds}. Note that this bound is independent 
of $v$, $x_1$ and $x_2$, and indeed also $w$, $u$ and $i$. It follows that for $y_1 \in \phi_{iw}(\bigcup_{n \geq 0} F_*^n)$ and 
$y_2 \in \phi_{iu}(\bigcup_{n \geq 0} F_*^n)$ there exist $x_1 \in F_{iw} \cap \bigcup_{n \geq 0} F^n$ and $x_2 \in F_{iu} \cap 
\bigcup_{n \geq 0} F^n$ such that $p(x_1) = y_1$ and $p(x_2) = y_2$ and so, using Proposition \ref{metriccvgce},
\begin{equation*}
d(y_1,y_2) = \lim_{v \to \infty} R_v(x_1,x_2) \geq c_{n+1} > 0.
\end{equation*}
Taking closures the same estimate holds for $y_1 \in \phi_{iw}(F_*)$ and $y_2 \in \phi_{iu}(F_*)$. Thus $\phi_{iw}(F_*)$ and 
$\phi_{iu}(F_*)$ are disjoint. Finally, if $y_1,y_2 \in F_*$ with $y_1 \neq y_2$ then by the contractive property there exists $n \geq 0$ 
and $w,u \in \Wbb_n$ such that $y_1 \in \phi_w(F_*)$ and $y_2 \in \phi_u(F_*)$ and $\phi_w(F_*), \phi_u(F_*)$ are disjoint. Then 
$\phi_{iw}(F_*), \phi_{iu}(F_*)$ are disjoint and so $\phi_i(y_1) \neq \phi_i(y_2)$.
\end{proof}

\begin{defn}
For $n \geq 0$ and $x \in F$ let
\begin{equation*}
D_n^0(x) = \bigcup \{ F_w: w \in \Wbb_n,\ F_w \ni x \},
\end{equation*}
\begin{equation*}
D_n^1(x) := \bigcup \{ F_w: w \in \Wbb_n,\ F_w \cap D_n^0(x) \neq \emptyset \}
\end{equation*}
be \textit{$n$-neighbourhoods} of $x$. These are each a sequence of decreasing subsets of $F$ (as $n$ increases). Let 
\begin{equation*}
\partial D_n^0(x) = (D_n^0(x) \cap F^n) \setminus \{ x \},
\end{equation*}
which is the boundary of $D_n^0(x)$ in the sense that any continuous path from an element of $D_n^0(x)$ to an element of $F \setminus D_n^0(x)$ must hit some element of $\partial D_n^0(x)$.

Similarly for $y \in F_*$ let
\begin{equation*}
D_n^0(y) = \bigcup \{ \phi_w(F_*): w \in \Wbb_n,\ \phi_w(F_*) \ni y \},
\end{equation*}
\begin{equation*}
D_n^1(y) := \bigcup \{ \phi_w(F_*): w \in \Wbb_n,\ \phi_w(F_*) \cap D_n^0(y) \neq \emptyset \}.
\end{equation*}
\end{defn}

\begin{prop}\label{pcts}
The projection map $p$ can be uniquely extended to a surjective function $p: \Gcal^v \to (F_*,d)$ (independently of $v$) which is continuous 
for all $v > v_{\min}$.
\end{prop}

\begin{proof}
Fix $v > v_{\min}$. Let $\epsilon > 0$. Then there exists $n \geq 0$ such that
\begin{equation*}
2(r_{\max}\rho_G^{-1})^n \diam(F_*) < \epsilon.
\end{equation*}
It follows that if $y \in F_*$ then $D_n^1(y) \subseteq B(y,\epsilon)$, the $\epsilon$-ball about $y$.
Using a method identical to that used in the proof of the previous proposition, we find that there is a constant $c_n > 0$ such that if 
$w,u \in \Wbb_n$ are such that $F_w, F_u$ are disjoint, then for all $z_1 \in F_w$ and $z_2 \in F_u$, $R_v(z_1,z_2) \geq c_n$. It follows that if $x \in F$ then
\begin{equation*}
B(x,c_n) \subseteq D^1_n(x).
\end{equation*}
Let $x \in \bigcup_{n \geq 0} F^n$ and $y = p(x) \in \bigcup_{n \geq 0} F_*^n$. We see that $p(D^1_n(x)) \subseteq D_n^1(y)$, so we may conclude that
\begin{equation*}
p(B(x,c_n)) \subseteq B(y,\epsilon).
\end{equation*}
The choice of $c_n$ did not depend on $x$, so $p$ is uniformly continuous on $x \in \bigcup_{n \geq 0} F^n$. It is a uniformly 
continuous function defined on a dense subset of a metric space $F$ and taking values in a complete space $F_*$, and so has a unique 
continuous extension to the whole of $F$. Now let $v_1, v_2 > v_{\min}$ and let $p_1,p_2 : F \to F_*$ be the functions generated by 
the above method by using $v = v_1, v_2$ respectively. By Theorem \ref{nfpDform} $\Gcal^{v_1}$ and $\Gcal^{v_2}$ have the same 
topology and so $p_1$ and $p_2$ are continuous functions on $\Gcal^{v_1}$ that agree on the dense subset $\bigcup_{n \geq 0} F^n$, 
and so must agree on the whole of $F$. So $p$ is independent of $v$.

It remains to prove the surjectivity of the extended $p$, so let $y \in F_*$. There exists a decreasing sequence of non-empty sets 
$(\phi_{w_n}(F_*))_n$ such that $w_n \in \Wbb_n$ and $y \in \phi_{w_n}(F_*)$ for all $n$. Each $\phi_{w_n}(F_*)$ is compact, hence closed, and
\begin{equation*}
\diam(\phi_{w_n}(F_*)) \to 0
\end{equation*}
by the contractivity of the $\phi_i$. Thus it must be the case that
\begin{equation*}
\{ y \} = \bigcap_{n=0}^\infty \phi_{w_n}(F_*),
\end{equation*}
and so
\begin{equation*}
p^{-1}(\{ y \}) = \bigcap_{n=0}^\infty p^{-1}(\phi_{w_n}(F_*))
\end{equation*}
where we note that the right-hand side is again a decreasing sequence of closed and non-empty sets. So $p^{-1}(\{ y \})$ must be non-empty.
\end{proof}

\begin{rem}
We may now naturally extend the equivalence relation $\sim$ to be defined on the whole of $F$. For $x,y \in F$, we define $x \sim y$ if and 
only if $p(x) = p(y)$.
\end{rem}

\begin{cor}\label{phip}
For all $1 \leq i \leq N$, $\phi_i \circ p = p \circ \psi_i$ on $F$. In particular, for any $w \in \Wbb_n$,
\begin{equation*}
p(F_w) = \phi_w(F_*).
\end{equation*}
\end{cor}

\begin{proof}
The maps $\phi_i \circ p$ and $p \circ \psi_i$ are continuous and agree on a dense subset of $F$.
\end{proof}

\subsubsection{Structural regularity of the projection map}

We seek to prove that $\Scal_* = (F_*, (\phi_i)_{1 \leq i \leq N})$ is a generalized p.c.f.s.s.~set and that its sequence of approximating networks (as in \eqref{Vn}) 
is $(F_*^n)_{n \geq 0}$. To do this we need a few more regularity results.
\begin{lem}\label{diamFw}
Let $n \geq 0$ and $w \in \Wbb_n$. Let $v_0 > v_{\min}$. If $v \geq v_0$ then taking $F_w$ as a subset of $\Gcal^v$ we have
\begin{equation*}
\diam (F_w,R_v) \leq r_{\max}^n\rho_n(v)^{-1} \sup_{v' \geq v_0} \diam \Gcal^{v'}.
\end{equation*}
\end{lem}
\begin{proof}
Recall from Lemma~\ref{diambound} that $\sup_{v \geq v_0} \diam \Gcal^v < \infty$. By the construction of the $\Gcal^v$ through 
the replication map, the function
\begin{equation*}
\psi_w: \Gcal^{\alpha^{-n}(v)} \to (F_w,R_v)
\end{equation*}
is a bijective contraction with Lipschitz constant at most $r_{\max}^n\rho_n(v)^{-1}$, which is strictly less than $1$ by 
Assumption~\ref{1paramass}. Thus
\begin{equation*}
\diam (F_w,R_v) \leq r_{\max}^n \rho_n(v)^{-1}\diam \Gcal^{\alpha^{-n}(v)} \leq r_{\max}^n \rho_n(v)^{-1} 
\sup_{v' \geq v_0} \diam \Gcal^{v'}.
\end{equation*}
\end{proof}
\begin{defn}
For $y \in F_*$ define $\Ccal(y)$ to be the closure of the set
\begin{equation*}
\left\{ z \in \bigcup_{n \geq 0} F^n: p(z) = y \right\}.
\end{equation*}
\end{defn}
Note that by Theorem \ref{nfpDform} all of the resistance metrics $R_v$ induce the same topology on $F$, so there is no confusion when talking about closures of subsets of $F$. Clearly if $x \in \Ccal(y)$ then $p(x) = y$, by continuity of $p$. The purpose of the next result is to prove a converse to this statement.
\begin{prop}\label{closure}
Let $y \in F_*$. Suppose there exists $z \in \bigcup_{n \geq 0} F^n$ such that $p(z) = y$. Then for all $x \in F$, if $p(x) = y$ then $x \in \Ccal(y)$.
\end{prop}
\begin{proof}
Suppose $x \in F$ such that $x \notin \Ccal(y)$. We aim to show that $p(x) \neq y$. If $x \in \bigcup_{n \geq 0} F^n$ then we may immediately conclude that $p(x) \neq y$, so we may assume that $x \notin \bigcup_{n \geq 0} F^n$. Fix a $v > v_{\min}$.

We have that $z \in \Ccal(y)$. The set $\Ccal(y)$ is a non-empty and closed subset of $F$, thus is compact, so $\epsilon := \inf_{x' \in \Ccal(y)}R_v(x, x') > 0$ is well-defined. Then by Lemma \ref{diamFw} there exists $n \geq 0$ such that $D^0_n(x) \subseteq B_v(x,\frac{\epsilon}{2})$, the open $R_v$-ball in $F$ with centre $x$ and radius $\frac{\epsilon}{2}$, and also such that $z \in F^n$. In particular,
\begin{equation}\label{nhoodnonint}
\left( D^0_n(x) \cap \bigcup_{m \geq 0} F^m \right) \cap \Ccal(y) = \emptyset,
\end{equation}
and notice that this is now independent of $v$. Now consider the indicator function of $\Ccal(y) \cap F^n$ as a function from $F^n$ to $\Rbb$, written as $\1bb_{\Ccal(y) \cap F^n}$. For $x_1,x_2 \in G^{v,n}$, if $\1bb_{\Ccal(y) \cap F^n}(x_1) \neq \1bb_{\Ccal(y) \cap F^n}(x_2)$ then the edge between $x_1$ and $x_2$ must either have zero conductance or have conductance of the form $r_w^{-1} \rho_n(v)$, by definition of $\Ccal(y)$. Then by Proposition \ref{CGexist}, for each $v_0 > v_{\min}$,
\begin{equation*}
\sup_{v \geq v_0}\Ecal^{(n)}_v(\1bb_{\Ccal(y) \cap F^n},\1bb_{\Ccal(y) \cap F^n}) < \infty.
\end{equation*}
We see by \eqref{nhoodnonint} that $\1bb_{\Ccal(y) \cap F^n}$ vanishes on $D^0_n(x) \cap F^n$. Therefore for any $v > v_{\min}$, the harmonic extension of $\1bb_{\Ccal(y) \cap F^n}$ to $\Gcal^v$ takes the value $1$ at $z$ and must take the value $0$ in all of $D^0_n(x)$, so it follows that for each $v_0 > v_{\min}$ there exists a constant $c > 0$ such that for all $x' \in D^0_n(x)$ and $v \geq v_0$,
\begin{equation*}
R_v(x',z) \geq c.
\end{equation*}
Now let $(x_i)_i$ be a sequence in $D^0_n(x) \cap \bigcup_{n \geq 0} F^n$ converging to $x$. For each $x_i$ we have $d(p(x_i),y) \geq c$ by Proposition \ref{metriccvgce}. Thus by continuity of $p$, $d(p(x),y) \geq c$. So $p(x) \neq y$.
\end{proof}
\begin{cor}\label{closure2}
Let $n \geq 0$, $w \in \Wbb_n$ and $x \in F_w$. Suppose there exists $z \in \bigcup_{n \geq 0} F^n$ such that $z \sim x$. Then there exists $z' \in F_w \cap \bigcup_{n \geq 0} F^n$ such that $z' \sim x$.
\end{cor}
\begin{proof}
By Proposition \ref{closure} there exists a sequence of $x_i \in \bigcup_{n \geq 0} F^n$ such that $x_i \sim x$ for all $i$ and $x_i \to x$ as $i \to \infty$. Suppose that the conclusion of the present result does not hold, then the sequence $(x_i)$ must accumulate (in a subsequence) in $F_v$ for some $v \in \Wbb_n$, $v \neq w$. This implies that $x \in F_w \cap F_v = \psi_w(F^0) \cap \psi_v(F^0)$ by \cite[Proposition 1.3.5(2)]{Kigami2001}, so $x \in F^n$ and we may simply take $z' = x$, which is a contradiction.
\end{proof}
Now we provide a partial extension of Proposition \ref{metriccvgce} to points in the space $F$.
\begin{lem}\label{metriccvgext}
Let $x,y \in F$. Then
\begin{equation*}
\limsup_{v \to \infty} R_v(x,y) \leq d(p(x),p(y)).
\end{equation*}
\end{lem}
\begin{proof}
Fix a $v_0 > v_{\min}$. Let $(x_i)_i$ and $(y_i)_i$ be sequences in $\bigcup_{n \geq 0}F^n$ such that $x_i \in D^0_i(x)$ and $y_i \in D^0_i(y)$ for each $i$. Evidently (Lemma \ref{diamFw}) $x_i \to x$ and $y_i \to y$ as $i \to \infty$. Let $\epsilon > 0$. Then there exists $i_0 \in \Nbb$ such that for all $i \geq i_0$ we have $\sup_{v \geq v_0} R_v(x_i,x) \leq \epsilon$ and $\sup_{v \geq v_0} R_v(y_i,y) \leq \epsilon$ (by Lemma \ref{diamFw}) and $d(p(x_i),p(y_i)) \leq d(p(x),p(y)) + \epsilon$ (by continuity of $p$). Thus
\begin{equation*}
\begin{split}
\limsup_{v \to \infty} R_v(x,y) &\leq \limsup_{v \to \infty}\left( R_v(x,x_i) + R_v(x_i,y_i) + R_v(y_i,y) \right)\\
&\leq 2\epsilon + \limsup_{v \to \infty}R_v(x_i,y_i).
\end{split}
\end{equation*}
Then by Proposition \ref{metriccvgce},
\begin{equation*}
\begin{split}
\limsup_{v \to \infty} R_v(x,y) &\leq 2\epsilon + d(p(x_i),p(y_i))
\leq 3\epsilon + d(p(x),p(y)).
\end{split}
\end{equation*}
Finally we take $\epsilon \to 0$ to get the required result.
\end{proof}
\begin{prop}\label{existdense}
Let $x,y \in F$ such that $x \neq y$ and $p(x) = p(y)$. Then there exists $z \in \bigcup_{n \geq 0}F^n$ such that $p(x) = p(z) = p(y)$.
\end{prop}

\begin{proof}
If $x \in \bigcup_{n \geq 0}F^n$ or $y \in \bigcup_{n \geq 0}F^n$, then the result instantly holds, so we assume that neither of these is the case. Since $x \neq y$, by Lemma \ref{diamFw} there exists $n_0 \geq 0$ such that for all $n \geq n_0$ we have $D^0_n(x) \cap D^0_n(y) = \emptyset$. We have $p(x) = p(y)$ so by Lemma \ref{metriccvgext}, $\lim_{v \to \infty} R_v(x,y) = 0$.

For $n \geq n_0$ and $v > v_{\min}$, we consider the distance $R_v(\partial D^0_n(x),\partial D^0_n(y))$. The sets $\partial D^0_n(x)$ and $\partial D^0_n(y)$ are both subsets of $F^n$, so it is enough approximate their effective resistance by considering functions from $F^n$ to $\Rbb$. Suppose $v > v_{\min}$, and suppose that the function $f:F^n \to \Rbb$ takes the value $0$ on $\partial D^0_n(x)$ and takes the value $1$ on $\partial D^0_n(x)$. Since $x,y \notin \bigcup_{n \geq 0}F^n$ we must have $\partial D^0_n(x) = D^0_n(x) \cap F^n$ and $\partial D^0_n(y) = D^0_n(y) \cap F^n$, so the harmonic extension of $f$ to $\Gcal^v$ must take the value $0$ at $x$ and the value $1$ at $y$. It follows that
\begin{equation*}
R_v(\partial D^0_n(x),\partial D^0_n(y)) \leq R_v(x,y)
\end{equation*}
for all $n \geq n_0$ and $v > v_{\min}$. In particular,
\begin{equation*}
\lim_{v \to \infty} R_v(\partial D^0_n(x),\partial D^0_n(y)) = 0
\end{equation*}
for all $n \geq n_0$. This means that for each $n \geq n_0$, on $G^{v,n}$ there exists some $x_n \in \partial D^0_n(x)$ and some $y_n \in \partial D^0_n(y)$ such that there is a path from $x_n$ to $y_n$ consisting only of edges with conductance of the form $r_w^{-1} \rho_n(v)\alpha^{-n}(v)$. Since $D^0_{n_0}(x) \cap D^0_{n_0}(y) = \emptyset$, this path (which depends on $n$ in general) must contain some $x'_n \in \partial D^0_{n_0}(x)$ and some $y'_n \in \partial D^0_{n_0}(y)$. Since the set $\partial D^0_{n_0}(x) \times \partial D^0_{n_0}(y)$ is finite we may take a subsequence $(x_{n_k},y_{n_k})$ of $(x_n,y_n)$ such that $x'_{n_k} =: x' \in \partial D^0_{n_0}(x)$ and $y'_{n_k} =: y' \in \partial D^0_{n_0}(y)$ are constant over $k$. Therefore
\begin{equation*}
x_{n_k} \sim x' \sim y' \sim y_{n_k}
\end{equation*}
for all $k$. Since $x_{n_k} \in D^0_{n_k}(x)$ we must have $x_{n_k} \to x$ as $k \to \infty$ by Lemma \ref{diamFw}, and likewise $y_{n_k} \to y$. So by the continuity of $p$,
\begin{equation*}
p(x) = p(x') = p(y') = p(y),
\end{equation*}
and $x',y' \in \bigcup_{n \geq 0} F^n$.
\end{proof}

We may now finally prove the target result.
\begin{thm}\label{S*}
$\Scal_* = (F_*, (\phi_i)_{1 \leq i \leq N})$ is a generalized p.c.f.s.s.~set and its sequence of approximating networks (as in \eqref{Vn}) 
is $(F_*^n)_{n \geq 0}$.
\end{thm}

\begin{proof}
We see that
\begin{equation*}
\begin{split}
B(\Scal_*) &= \bigcup_{i \neq j} \left(\phi_i \left( F_* \right) \cap \phi_j \left( F_* \right) \right)\\
&= \bigcup_{i \neq j} \left(p \left( F_i \right) \cap p \left( F_j \right) \right)
\end{split}
\end{equation*}
and so $p(B(\Scal)) \subseteq B(\Scal_*)$. If $y \in B(\Scal_*)$, then let $i \neq j$ be such that $y \in \phi_i \left( F_* \right) \cap \phi_j 
\left( F_* \right)$. Then there exist $x_i \in F_i$ and $x_j \in F_j$ such that $p(x_i) = y = p(x_j)$. If $x_i = x_j$, then $x_i \in F_i \cap F_j$ 
and so $y \in p(B(\Scal))$. On the other hand if $x_i \neq x_j$, then $x_i \sim x_j$ and so by Proposition \ref{existdense} there exists $z \in \bigcup_{n\geq 0}F^n$ such that $z \sim x_i \sim x_j$. Then by Corollary \ref{closure2} there must exist $x'_i \in F_i \cap \bigcup_{n\geq 0}F^n$ and $x'_j \in F_j \cap \bigcup_{n\geq 0}F^n$ such that
\begin{equation*}
x'_i \sim x_i \sim x_j \sim x'_j.
\end{equation*}
Choose $n \geq 0$ such that $x'_i,x'_j \in F^n$. Then in $G^{v,n}$ there must exist a path between $x'_i$ and $x'_j$ consisting only of edges with conductance $r_w^{-1}\rho_n(v)\alpha^{-n}(v)$. Since $i \neq j$, this path must contain some element $z$ of $F_i \cap F_k$ for some $k \in \{1,\ldots,N\}$ such that $k \neq i$. So $z \in B(\Scal)$ and $p(z) = y$, so $y \in p(B(\Scal))$. Thus
\begin{equation*}
p(B(\Scal)) = B(\Scal_*).
\end{equation*}
Now it follows from the surjectivity of $p$ that
\begin{equation}\label{piPS}
\begin{split}
\pi(P(\Scal_*)) &= \left\{ x \in F_*: \exists w \in \bigcup_{n \geq 1} \Wbb_n,\ \phi_w(x) \in B(\Scal_*) \right\}\\
&= \left\{ p(x) \in F_*: x \in F,\ \exists w \in \bigcup_{n \geq 1} \Wbb_n,\ p \circ \psi_w(x) \in B(\Scal_*) \right\}\\
&= \left\{ p(x) \in F_*: x \in F,\ \exists w \in \bigcup_{n \geq 1} \Wbb_n,\ \exists y \in B(\Scal),\ \psi_w(x) \sim y \right\}\\
&= p \left( \left\{ x \in F: \exists w \in \bigcup_{n \geq 1} \Wbb_n,\ \exists y \in B(\Scal),\ \psi_w(x) \sim y \right\} \right).
\end{split}
\end{equation}
Recall that
\begin{equation*}
\pi(P(\Scal)) = \left\{ x \in F: \exists w \in \bigcup_{n \geq 1} \Wbb_n,\  \psi_w(x) \in B(\Scal) \right\},
\end{equation*}
so it is clear that $p(\pi(P(\Scal))) \subseteq \pi(P(\Scal_*))$.

Now suppose that $x$ is a member of the subset of $F$ described by the last line of \eqref{piPS}. That is, there exists $n \geq 1$ and $w \in \Wbb_n$ and $y \in B(\Scal)$ such that $\psi_w(x) \sim y$. We observe that $y \in F^1$.

Assume that $x \notin \pi(P(\Scal)) = F^0$. Since $n \geq 1$, if $y \in F_w$ then $y \in \psi_w(F^0)$ so there exists some $z \in F^0$ such that $p(\psi_w(z)) = p(\psi_w(x))$. We now suppose that $y \notin F_w$. By Corollary \ref{closure2} there exists $x' \in F_w \cap \bigcup_{n \geq 0} F^n$ such that $x' \sim \psi_w(x) \sim y$. Pick $m \geq n$ such that $x',y \in F^m$. Then there exists a path in $G^{v,m}$ from $x'$ to $y$ consisting only of edges with conductance of the form $r_{w'}^{-1}\rho_m(v)\alpha^{-m}(v)$. Since $x' \in F_w$ and we have assumed that $y \notin F_w$, this path must contain some element of $\psi_w(F^0)$. So there exists some $z \in F^0$ such that $p(\psi_w(z)) = p(\psi_w(x))$.

We have seen that in every case there exists $z \in F^0 = \pi(P(\Scal))$ such that $p(\psi_w(z)) = p(\psi_w(x))$. So $\phi_w(p(z)) = \phi_w(p(x))$ and so $p(z) = p(x)$ by the injectivity of the $\phi_i$ (Proposition \ref{phiinject}). We conclude that
\begin{equation*}
p \left( \left\{ x \in F: \exists w \in \bigcup_{n \geq 1} \Wbb_n,\ \exists y \in B(\Scal),\ \psi_w(x) \sim y \right\} \right) \subseteq p(\pi(P(\Scal))),
\end{equation*}
and so
\begin{equation*}
\pi(P(\Scal_*)) = p(\pi(P(\Scal))) = p(F^0) = F^0_*.
\end{equation*}
This is finite, so $\Scal_*$ is a generalized p.c.f.s.s.~set. Then \eqref{F*approx} implies that $(F_*^n)_{n \geq 0}$ is the 
associated sequence of approximating networks.
\end{proof}

\subsection{The limit process}

Now let $\nu$ be the pushforward measure of $\mu$ onto $F_*$ through the continuous function $p$. It is easy to verify that $\nu$ is 
a Bernoulli measure and that the $\nu_n$ are its approximations as given in Definition~\ref{approxmeasure}.
\begin{thm}\label{F*Dform}
$\Ecal^{(0)}_*$ is a regular non-degenerate fixed point of the renormalization operator $\Lambda_*$ of $\Scal_*$ with respect to the 
resistance vector $r$. Let
\begin{equation*}
\begin{split}
&D_* = \left\{ f \in C(F_*): \sup_{n} \Ecal^{(n)}_*(f,f) < \infty \right\},\\
&\Ecal_*(f,f) = \sup_{n} \Ecal^{(n)}_*(f,g),\quad f,g \in D_*.
\end{split}
\end{equation*}
Then the pair $(\Ecal_*,D_*)$ is a regular local irreducible Dirichlet form on $\Lcal^2(F_*,\nu)$.

If $X^{*,n}$ is the Markov process associated with $H^n = (F_*^n,\Ecal^{(n)}_*)$ and $\Lcal^2(F^n_*, \nu_n)$ for each $n$, and if $X^*$ 
is the Markov process associated with $(F_*,\Ecal_*)$ and $\Lcal^2(F_*, \nu)$, then $X^{*,n} \to X^*$ weakly in $\Dcal_{F_*}[0,\infty)$.
\end{thm}
\begin{proof}
$\Lambda_*$ is defined on the set of conservative Dirichlet forms on $F^0_*$. Let $R_*$ be the associated replication operator. We see that
\begin{equation*}
\begin{split}
R_*(\Ecal^{(0)}_*)(f,g) &= \sum_{i=1}^N r_i^{-1} \Ecal^{(0)}_*(f \circ \phi_i , g \circ \phi_i)\\
&= \sum_{i=1}^N r_i^{-1} \Ecal^{(0)}_\infty (f \circ \phi_i \circ p , g \circ \phi_i \circ p)\\
&= \sum_{i=1}^N r_i^{-1} \Ecal^{(0)}_\infty (f \circ p \circ \psi_i , g \circ p \circ \psi_i)\\
&= \lim_{v \to \infty} R(\Ecal^{(0)}_v)(f \circ p, g \circ p)
\end{split}
\end{equation*}
where $R$ is the replication operator of $\Scal$. Then
\begin{equation*}
\begin{split}
\Lambda_*(\Ecal^{(0)}_*)(f,f) &= \inf\{ R_*(\Ecal^{(0)}_*)(g,g) : g \in C(F^1_*),\ g|_{F^0_*} = f \}\\
&= \inf\{ \lim_{v \to \infty} R(\Ecal^{(0)}_v)(g \circ p, g \circ p) : g \in C(F^1_*),\ g|_{F^0_*} = f \}\\
&= \inf\{ \lim_{v \to \infty} R(\Ecal^{(0)}_v)(g, g) : g \in C_*(F^1),\ g|_{F^0} = f \circ p \}
\end{split}
\end{equation*}
Now we observe that for any $g \in C(F^1) \setminus C_*(F^1)$, we have $\lim_{v \to \infty} R(\Ecal^{(0)}_v)(g, g) = \infty$. Also 
we can restrict the set over which we take the infimum to functions taking values in $[\min f, \max f]$. This is a compact set, so we 
can use Lemma~\ref{topo}. So
\begin{equation*}
\begin{split}
\Lambda_*(\Ecal^{(0)}_*)(f,f) &= \lim_{v \to \infty} \inf\{ R(\Ecal^{(0)}_v)(g, g) : g \in C(F^1),\ g|_{F^0} = f \circ p \}\\
&= \lim_{v \to \infty} \Lambda(\Ecal^{(0)}_v)(f \circ p, f \circ p)\\
&= \lim_{v \to \infty} (\rho(v)^{-1} \Ecal^{(0)}_{\alpha(v)}(f \circ p, f \circ p))\\
&= \rho_G^{-1} \Ecal^{(0)}_\infty(f \circ p, f \circ p)\\
&= \rho_G^{-1} \Ecal^{(0)}_*(f, f)
\end{split}
\end{equation*}
and so $\Ecal^{(0)}_*$ is a regular non-degenerate fixed point of $\Lambda_*$ with eigenvalue $\rho_G^{-1}$. It is then simple to verify that
\begin{equation*}
\Ecal^{(n)}_*(f,g) := \rho_G^n\sum_{w \in \Wbb_n} r_w^{-1} \Ecal^{(0)}_*(f \circ \phi_w , g \circ \phi_w), \quad f,g \in C(F_*^n)
\end{equation*}
and so we are able to use Theorems~\ref{fpDform} and~\ref{fpcvgce}.
\end{proof}

\section{Convergence of processes}

We continue with the set-up of the previous sections; We have $\Scal = (F,(\psi_i)_{1 \leq i \leq N})$ a connected generalized p.c.f.s.s.~set with a one-parameter iterable family $\Dbb$ (Assumption \ref{1paramass}, Assumption \ref{alphalimass}) satisfying $\Dbb$-injectivity (Assumption \ref{similass}). We have the spaces $\Gcal^v = (F,\Ecal_v)$ and the limit space $(F_*, \Ecal_*)$ constructed in the previous section. Recall that $H^n = (F_*^n , \Ecal^{(n)}_*)$. 
Let $\Hcal = (F_*, \Ecal_*)$. Recall that $X^v$ is the diffusion associated with $\Gcal^v$ and $X^*$ is the diffusion associated with $\Hcal$.

\subsection{The ambient space}

Fix a $v_0 > v_{\min}$. We aim to construct a metric space $\Mcal = (E,d)$ such that all of our metric spaces $(\Gcal^v)_{v \geq v_0}$ 
and $\Hcal$ exist as disjoint subspaces of $\Mcal$ in such a way that $\Gcal^v \to \Hcal$ as $v \to \infty$ in the Hausdorff topology on 
$\Mcal$. This will mean that we can define all of the $X^v$ and $X^*$ on the same space and so it will be easier to study the relationship 
between them. The new metric $d$ will agree with the old definition of $d$ as the resistance metric on $\Hcal$, so no confusion will occur. 
We construct the underlying set $E$ as follows.

Let $E$ be the disjoint union
\begin{equation*}
E = \Hcal \sqcup \bigsqcup_{v \geq v_0} \Gcal^v.
\end{equation*}
By construction $E$ contains each $\Gcal^v$ and $\Hcal$. Notice that $E$ also contains each $G^{v,n}$ as a subspace of its 
respective $\Gcal^v$ and likewise contains each $H^n$ as a subspace of $\Hcal$.

We also perform our last extension of the projection map $p$. We will now have
\begin{equation*}
p: \Mcal \to \Hcal
\end{equation*}
where $p(x)$ agrees with the previous definition if $x \in \Gcal^v$ for any $v \geq v_0$, and $p(x) = x$ if $x \in \Hcal$. We also define 
restrictions of this new $p$ to certain subspaces of $E$: for $v \geq v_0$ and $n \geq 0$ let the functions $p_{v,n}$ and $p_v$ be given by
\begin{equation*}
\begin{split}
p_{v,n} &= p|_{G^{v,n}}: G^{v,n} \to H^n,\\
p_v &= p|_{\Gcal^v}: \Gcal^v \to \Hcal.
\end{split}
\end{equation*}

\subsubsection{Constructing the metric}\label{metricconstruct}

We start with an extension of Lemma \ref{diamFw}.
\begin{lem}\label{diamH}
$\diam \Hcal < \infty$. Let $n \geq 0$ and $w \in \Wbb_n$. Then
\begin{equation*}
\diam (\phi_w(F_*)) \leq r_{\max}^n\rho_G^{-n} \diam \Hcal.
\end{equation*}
\end{lem}
\begin{proof}
$\Hcal$ is a self-similar structure so is compact, so $\diam \Hcal < \infty$. This proof is essentially the same as the proof of Lemma \ref{diamFw}. The function $\phi_w$ on $\Hcal$ is a 
bijective contraction with Lipschitz constant at most $r_{\max}^n\rho_G^{-n}$, which is strictly less than $1$ by 
Assumption~\ref{1paramass}. So
\begin{equation*}
\diam \phi_w(F_*) \leq r_{\max}^n\rho_G^{-n}\diam \Hcal.
\end{equation*}
\end{proof}
\begin{defn}
Let $\delta_v$ be the Hausdorff metric on non-empty compact subsets of $\Gcal^v$. Let $\delta_*$ be the Hausdorff metric on 
non-empty compact subsets of $\Hcal$.
\end{defn}

\begin{cor}\label{Hausbd}
For all $n \geq 0$,
\begin{equation*}
\sup_{v \geq v_0}\delta_v(G^{v,n},\Gcal^v) \leq r_{\max}^n\rho_n(v)^{-1} \sup_{v \geq v_0} \diam \Gcal^{v}
\end{equation*}
and
\begin{equation*}
\delta_*(H^n,\Hcal) \leq r_{\max}^n\rho_G^{-n} \diam \Hcal.
\end{equation*}
\end{cor}

\begin{proof}
We observe that $G^{v,n}$ is a metric subspace of $\Gcal^v$. If $x \in \Gcal^v$ then there exists $w \in \Wbb_n$ such that $x \in F_w$. Then 
there exists $y \in G^{v,n}$ such that $y \in F_w$ also. Then by Lemma~\ref{diamFw},
\begin{equation*}
R_v(x,y) \leq r_{\max}^n\rho_n(v)^{-1} \sup_{v' \geq v_0} \diam \Gcal^{v'}.
\end{equation*}
It follows that
\begin{equation*}
\delta_v(G^{v,n},\Gcal^v) \leq r_{\max}^n\rho_n(v)^{-1} \sup_{v' \geq v_0} \diam \Gcal^{v'}.
\end{equation*}
The second statement is proven similarly using Lemma \ref{diamH}.
\end{proof}

We now have all the pieces we need to construct $d$.
\begin{defn}[The metric $d$ on $\Mcal$]\label{Mmetric}
If $x,y \in \Gcal^v$ for some $v \geq v_0$ then define $d(x,y) = R_v(x,y)$. On $\Hcal$ likewise we define $d$ to agree with the 
existing resistance metric (which was also called $d$).

Let $n_0 = 0$. For each $k \geq 1$, using Corollary~\ref{Hausbd}, Lemma~\ref{Cbounds} and regularity (Assumption~\ref{1paramass}) 
we can choose a $n_k \geq 0$ such that for all $n \geq n_k$,
\begin{equation*}
\begin{split}
\sup_{v  \geq v_0}\delta_v(G^{v,n},\Gcal^v) &\leq \frac{1}{2k},\\
\delta_*(H^n, \Hcal) &\leq \frac{1}{2k}.
\end{split}
\end{equation*}
Without loss of generality we may assume that the sequence $(n_k)_{k \geq 0}$ is strictly increasing.

Let $c = \sup_{v \geq v_0} \dis p_{v,0}$, which is finite by Proposition~\ref{metriccvgce} and the fact that the resistance metric $R_v$ 
restricted to $F^0$ is continuous in $v$ (using Lemma~\ref{topo}). Now for each $k \geq 1$ we use the finiteness of $F^{n_k}$ and 
Proposition~\ref{metriccvgce} to find a $v_k \geq v_0$ such that for all $v \geq v_k$,
\begin{equation*}
\dis p_{v,n_k} \leq \frac{c}{k+1}.
\end{equation*}
Again without loss of generality we may assume that $v_{k+1} \geq v_k + 1$ for all $k \geq 0$.

Let $v \geq v_0$. By the construction of the sequence $(v_k)_{k \geq 0}$ we may pick the unique $k \geq 0$ such that $v_k \leq v < v_{k+1}$. 
For this $k$, for each $x \in G^{v,n_k} \subseteq \Gcal^v$ we set
\begin{equation}\label{dxpx}
d(x,p_{v,n_k}(x)) = \frac{c}{2(k+1)}.
\end{equation}
The restriction of $d$ to the spaces $\Gcal^v$ and $\Hcal$ is then constructed naturally around this: For $x \in \Gcal^v$ and $y \in \Hcal$, set
\begin{equation*}
d(x,y) = \inf_{x' \in G^{v,n_k}} \left\{ d(x,x') + \frac{c}{2(k+1)} + d(p_{v,n_k}(x'),y) \right\}.
\end{equation*}
The bound on the distortion of $p_{v,n_k}$ ensures that this is still a metric. Finally, for $v_1 \neq v_2$, $x_1 \in \Gcal^{v_1}$, 
$x_2 \in \Gcal^{v_2}$, we define
\begin{equation*}
d(x_1,x_2) = \inf_{y \in \Hcal} \left\{ d(x_1,y) + d(y,x_2) \right\}.
\end{equation*}
This is indeed a metric. Checking that the triangle inequality holds is a simple but tedious exercise.
\end{defn}

\subsection{Convergence}
\begin{defn}
For each $v \geq v_0$ let $\mu^v$ be the Borel probability measure on $\Mcal$ given by
\begin{equation*}
\mu^v(A) = \mu(A \cap \Gcal^v),
\end{equation*}
where $\mu$ is interpreted as a measure on $\Gcal^v = (F,R_v)$. That is, $\mu^v$ is the measure $\mu$ defined on the $v$th copy of $F$ 
in $E$.
\end{defn}
\begin{prop}\label{GHP}
$\Gcal^v \to \Hcal$ as $v \to \infty$ in the Hausdorff metric on non-empty compact subsets of $\Mcal$. Additionally, $\mu^v \to \nu$ as 
$v \to \infty$ weakly as probability measures on $\Mcal$.
\end{prop}
\begin{proof}
Let $\delta$ be the Hausdorff metric on subsets of $\Mcal$. We take the sequences $(n_k)_{k \geq 0}$ and $(v_k)_{k \geq 0}$ and the 
constant $c > 0$ from Definition \ref{Mmetric}. Then for each $k \geq 0$ we have $n_k$ satisfying
\begin{equation*}
\begin{split}
\sup_{v \geq v_0}\delta(G^{v,n_k},\Gcal^v) &\leq \frac{1}{2k},\\
\delta(H^{n_k}, \Hcal) &\leq \frac{1}{2k}.
\end{split}
\end{equation*}
In addition, for all $v > v_k$ we have
\begin{equation*}
\delta(G^{v,n_k},H^{n_k}) \leq \frac{c}{2(k+1)}
\end{equation*}
by \eqref{dxpx}. Recall that we constructed the sequence $(v_k)_k$ to be strictly increasing with $\lim_{k \to \infty}v_k = \infty$, and so
\begin{equation*}
\lim_{v \to \infty} \delta(\Gcal^v,\Hcal) = 0.
\end{equation*}

We now tackle the weak convergence result. Since $\nu$ is the push-forward measure of $\mu^v$ with respect to $p_v$ for all 
$v \geq v_0$, it suffices to prove that
\begin{equation*}
\lim_{v \to \infty}\sup_{x \in \Gcal^v} d(x,p(x)) = 0.
\end{equation*}
Let $v \geq v_0$ and $x \in \Gcal^v$. There exists a unique $k \geq 0$ such that $v_k \leq v < v_{k+1}$. Let $w \in \Wbb_{n_k}$ 
such that $x \in F_w \subseteq \Gcal^v$, and pick some $y \in (G^{v,n_k} \cap F_w) \subseteq \Gcal^v$. Then Corollary~\ref{phip} 
implies that $p(x), p(y) \in \phi_w(F_*)$. Then by Lemma~\ref{diamFw}, Lemma \ref{diamH} and the construction of the metric $d$ we have that
\begin{equation*}
\begin{split}
d(x,p(x)) &\leq d(x,y) + d(y,p(y)) + d(p(y),p(x))\\
&\leq r^{n_k}_{\max}\rho_G^{-n_k}(k_\rho^{-1} \sup_{v' \geq v_0} \diam \Gcal^{v'} + \diam \Hcal) + \frac{c}{2(k+1)}.
\end{split}
\end{equation*}
Therefore
\begin{equation*}
\sup_{x \in \Gcal^v} d(x,p(x)) \leq r^{n_k}_{\max}\rho_G^{-n_k}(k_\rho^{-1} \sup_{v' \geq v_0} \diam \Gcal^{v'} + \diam \Hcal) + 
\frac{c}{2(k+1)}.
\end{equation*}
Proposition~\ref{CGexist} states that $r_{\max}\rho_G^{-1} < 1$. Lemma~\ref{diambound} and Lemma~\ref{diamH} respectively 
state that $\sup_{v' \geq v_0} \diam \Gcal^{v'} < \infty$ and $\diam \Hcal < \infty$. By the construction of the metric we have that 
$n_k \to \infty$ and $v_k \to \infty$ as $k \to \infty$. The choice of $k$ is such that $k \to \infty$ as $v \to \infty$. Therefore
\begin{equation*}
\lim_{v \to \infty}\sup_{x \in \Gcal^v} d(x,p(x)) = 0.
\end{equation*}
\end{proof}
We are now in a position to state and prove the main theorem of the paper. All of the assumptions needed are described at the start of this section.

\begin{thm}\label{thm:main}
Let $x \in F_*$ and for each $v \geq v_0$ let $x_v \in p^{-1}(x) \subseteq \Gcal^v$. Then
\begin{equation*}
\left( X^v, \Pbb^{x_v} \right) \to \left( X^*, \Pbb^x \right)
\end{equation*}
as $v \to \infty$ weakly as random variables in $\Dcal_\Mcal[0,\infty)$, the Skorokhod space of c\`{a}dl\`{a}g processes on $\Mcal$.
\end{thm}

\begin{proof}
It suffices to prove convergence on a countable sequence $(u_m)_{m=0}^\infty$ such that $u_m \geq v_0$ for all $m$ and 
$\lim_{m \to \infty}u_m = \infty$. Let
\begin{equation*}
\Mcal' := \Hcal \sqcup \bigsqcup_{m \geq 0} \Gcal^{u_m} \subset \Mcal,
\end{equation*}
in which $\Hcal$ and all the $\Gcal^{u_m}$ are isometrically embedded. This inherits the metric $d$ from $\Mcal$ but is much more 
manageable, indeed it is compact. We prove sequential compactness. Let $(x_i)_i$ be a sequence in $\Mcal'$. If an infinite number of 
the $x_i$ lie in a single $\Gcal^{u_m}$ or in $\Hcal$, then all of these individual subspaces are compact so we can take a convergent 
subsequence. Therefore assume that this is not the case. By taking a subsequence, we can assume that for each $i$, $x_i \in \Gcal^{u_{m_i}}$ 
where $(u_{m_i})_i$ is a strictly increasing sequence. Observe that $\lim_{i \to \infty}u_{m_i} = \infty$. For each $i$, using the sequences 
$(n_k)_k$ and $(v_k)_k$ and the constant $c > 0$ defined in Definition~\ref{Mmetric}, pick the unique $k_i \geq 0$ such that 
$v_{k_i} < u_{m_i} \leq v_{k_i + 1}$. The sequence $(k_i)_i$ is thus non-decreasing and increases up to infinity. For each $i$, pick 
$w_i \in \Wbb_{n_{k_i}}$ such that $x_i \in F_{w_i}$ and then pick some $y_i \in F^{n_{k_i}}_{w_i} \subseteq \Gcal^{u_{m_i}}$. 
Then $d(x_i,y_i) \leq \diam (F_{w_i}, R_{u_{m_i}})$ and $d(y_i,p(y_i)) = \frac{c}{2(k_i+1)}$, by the construction of $\Mcal$ in 
Definition~\ref{Mmetric}. By Lemma~\ref{Cbounds}, Lemma~\ref{diamFw} and the fact that $n_{k_i} \to \infty$, we thus have that
\begin{equation*}
\lim_{i \to \infty}d(x_i,p(y_i)) \leq \lim_{i \to \infty} \left( k_\rho^{-1}(r_{\max}\rho_G^{-1})^{n_{k_i}} \sup_{v' \geq v_0} \diam \Gcal^{v'} 
+ \frac{c}{2(k_i+1)} \right) = 0.
\end{equation*}
Now $(p(y_i))_i$ is a sequence in $\Hcal$, which is a compact metric space, and so has a subsequential limit. By taking a subsequence, 
we can assume that $p(y_i) \to z \in \Hcal$. Thus $x_i \to z$ as well, and so we have sequential compactness.

Now to prove convergence of processes we use \cite[Theorem 7.1]{croydon2016a}. The method of proof of weak convergence in 
Proposition~\ref{GHP} implies that $x_{u_m} \to x$ as $m \to \infty$ in $\Mcal'$. Then Proposition~\ref{GHP} and the compactness 
of the spaces $\Gcal^{u_m}$ and $\Hcal$ immediately implies that $(\Gcal^{u_m},R_{u_m},\mu^{u_m},x_{u_m}) \to (\Hcal,d,\nu,x)$ 
as $m \to \infty$ in the \textit{spatial Gromov-Hausdorff-vague} topology (see \cite[Section 7]{croydon2016a}) on the compact space 
$\Mcal'$. Thus the conditions of \cite[Theorem 7.1]{croydon2016a} are satisfied (see \cite[Remark 1.3(b)]{croydon2016a}), and so 
we conclude that
\begin{equation*}
\left( X^{u_m}, \Pbb^{x_{u_m}} \right) \to \left( X^*, \Pbb^x \right)
\end{equation*}
as $m \to \infty$ as random elements of $\Dcal_{\Mcal'}[0,\infty)$, the Skorokhod space of c\`{a}dl\`{a}g processes on $\Mcal'$. 
Therefore this convergence also occurs in $\Dcal_\Mcal[0,\infty)$.
\end{proof}

\subsection{The Sierpinski gasket}

Finally we give the proof of the result stated in the introduction for the Sierpinski gasket.
Calculations with traces give that, for the Sierpinski gasket with conductances as in Figure~\ref{fig:gasblob},
\begin{equation}
\alpha(v) = (3v^2+6v+1)/(4v+6)\;\;\mbox{and} \;\; \rho(v) = (3v+2)/(2v+1), \label{eq:gasmap}
\end{equation}
and hence $\rho_G=3/2$ and $\beta=4/3$ giving a one-parameter iterable family which is also $\mathbb{D}$-injective. As $v\to\infty$, by Theorem~\ref{S*}, the Sierpinski gasket becomes the shorted gasket of Figure~\ref{fig:gasblob}.


\begin{proof}[Proof of Theorem~\ref{thm:sg}]
%
Recall that we have $X^{a,v}_0=p_2=(1/2,\sqrt{3}/2)$.
We let $L=L_0$ denote the line segment from $p_1=(0,0)$ to $p_3=(1,0)$ and let 
\[ L_n = \psi_2^n(L), \]
denote the line segment which forms the base of the triangle of side $2^{-n}$, with top vertex at $p_2$. 
Now let $T_{L_n}(X) = \inf\{t> 0: X_t \in L_n\}$, that is
the hitting time of the base of the top triangle. 
If we set 
\[ \tilde{X}^{v,n}_t = \psi_2^{-n}(X^v_t)), \;\; 0\leq t\leq T_{L_n}(X^{v}), \]
then we have a process on the Sierpinski gasket, $G$, run until it hits the base line  $L$, 
which has the same paths as the process in which the effective conductors on the graph $G_0$ associated to $G$ are given by 
$(\rho_n(v), \rho_n(v), \rho_n(v)\alpha^{-n}(v))$. Thus it is the same process in law as $X^{\alpha^{-n}(v)}_t, 0\leq t\leq 
T_L(X^{\alpha^{-n}(v)})$ up to a time change. This time change is $3^n\rho_n(v)$ to take into account the resistance scaling 
and the scaling in the invariant measure. Hence the law of $\tilde{X}^{v,n}_{t / 3^n\rho_n(v)}=X^{\alpha^{-n}(v)}_t$ for 
$0\leq t\leq T_L(X^{\alpha^{-n}(v)})$.

We recall that $1$ denotes the base point of the shorted gasket.
From \eqref{eq:gasmap} we can apply Theorem \ref{thm:main} to see that $X^{\alpha^{-n}(v)} \to X^s$ weakly giving the result that
\[ \{ \psi_2^{-n}(X^v_{t/3^n\rho_n(v)}), \;\; 0\leq t\leq T_{L_n}(X^{v})/3^n\rho_n(v) \} \to \{X^s_t, \;\; 0\leq t \leq T_1(X^s) \}, \]
weakly in $\Dcal_\Mcal[0,\infty)$.

Finally we observe that these are deterministic time changes and by \eqref{eq:rholim} we have the existence of a constant 
$\sigma = \lim_{n\to\infty} 3^n\rho_n(v)/3^n\rho_G^n>0$. Hence, as our processes are almost surely continuous, 
we have the weak convergence 
\[ \{ \psi_1^{-n}(X^v_{t/3^n\rho_G^n}), \;\; 0\leq t\leq T_{L_n}(X^{v})/3^n\rho_G^n \} \to \{X^s_{\sigma t}, 
\;\; 0\leq t \leq T_1(X^s) \}. \]
As $\rho_G=3/2$ the ultraviolet scaling factor is given by $\lambda=3\rho_G=9/2$ as required.
\end{proof}

\begin{rem}
We make the following observations:
\begin{enumerate}
\item It is straightforward to extend the analysis here to many fractals based on the $d$-dimensional tetrahedron. For instance
the examples mentioned in \cite{hambly1998}.
\item In the case of the (weighted) Vicsek set,  with resistance weights $r=(1,1,1,1,s)$ as shown in Figure~\ref{fig:vicsek}, the analysis does not hold.  Calculations with traces give that
\[ \rho_s(v) =s+\frac{4}{1+v}, \;\; \alpha_s(v)=\frac{sv+2}{s+2}. \]
Thus $\rho_G=s, \beta=1+2/s$ and, even though $\beta>1$ for all $s>0$, we can see that this is not asymptotically regular as 
$\rho_G^{-1} r_{\max} \geq 1$ for any $s>0$ (the $s=1$ case was discussed in \cite{hambly1998}).
\item In \cite{hattori97} a version of the asymptotically one-dimensional process is constructed on a 
scale irregular Sierpinski gasket. This class of fractals was analysed in \cite{ham92, barham97}, where
an analogue of the fixed point diffusion was constructed. Our approach developed here could be applied to
yield a scale irregular `shorted gasket' as the short time asymptotic limit for the asymptotically one-dimensional
process on the scale irregular Sierpinski gasket.
\item The spectral dimension of the limit fractal for the Sierpinski gasket is $\log{9}/\log{9/2}$ which is the local spectral dimension of the
asympotically one-dimensional diffusion on the Sierpinski gasket~\cite{hambly2002}. We conjecture that the local spectral dimension for 
the non-fixed point diffusion should be the spectral dimension for the limiting fractal in general.
\end{enumerate}
\end{rem}

\begin{figure}[ht]
\centerline{\includegraphics[height=1in]{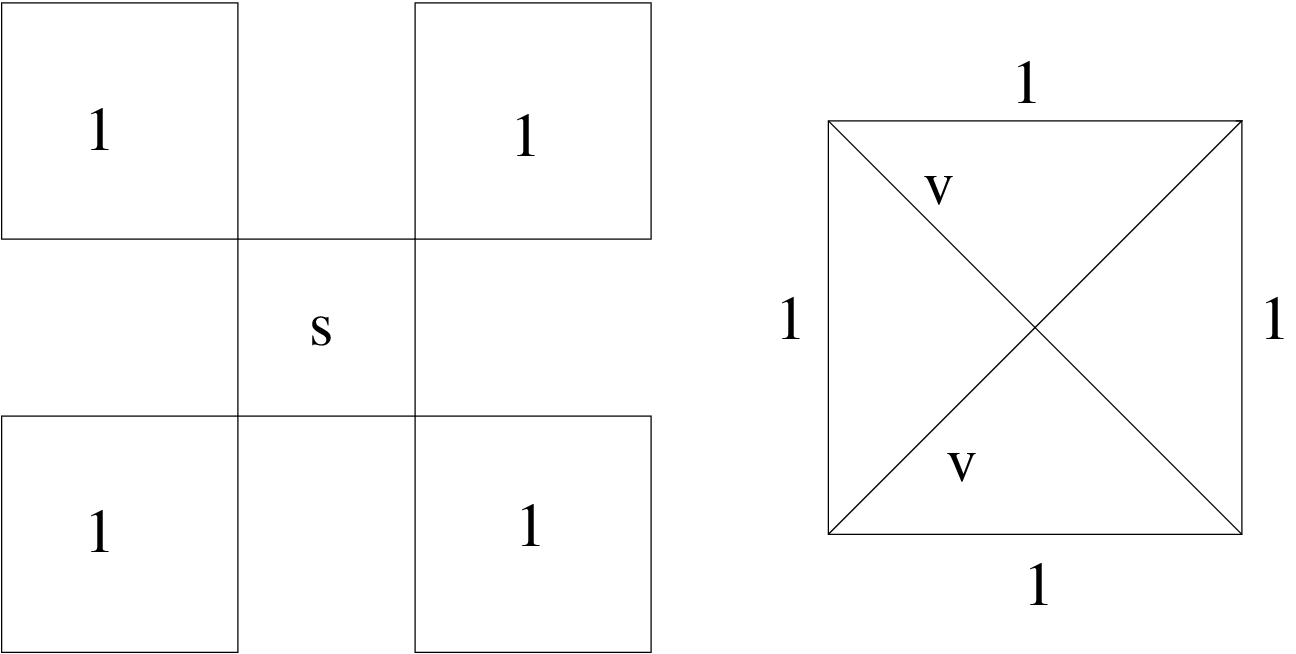}}
\caption{The Vicsek set with resistance weights $r=(1,1,1,1,s)$ is not asymptotically regular for any $s>0$.}
\label{fig:vicsek}
\end{figure}  

\bibliographystyle{plain}


\end{document}